\documentclass[10pt]{article}
\usepackage[margin=1in]{geometry} 
\usepackage{libertine}
\usepackage[libertine]{newtxmath}

\usepackage[english]{babel} 
\usepackage{amsmath, amsfonts, amsthm, amssymb} 
\usepackage[scr=boondoxo,cal=euler]{mathalfa} 
\usepackage{tikz-cd} 
\usepackage{indentfirst} 
\usepackage{fancyhdr} 
\usepackage{graphicx} 
\usepackage{enumitem} 
\usepackage{microtype} 
\usepackage{pdfpages} 
\usepackage{centernot} 
\usepackage{ragged2e} 
\usepackage{pinlabel}
\usepackage{caption}
\usepackage[hidelinks]{hyperref} 

\allowdisplaybreaks 
\usetikzlibrary{decorations.pathmorphing}

\theoremstyle{plain}
\newtheorem{thm}{Theorem}[section]
\newtheorem*{thm*}{Theorem}
\newtheorem{lem}[thm]{Lemma}
\newtheorem*{lem*}{Lemma}
\newtheorem{cor}[thm]{Corollary}
\newtheorem*{cor*}{Corollary}
\newtheorem{prop}[thm]{Proposition}
\newtheorem*{prop*}{Proposition}

\newtheorem*{conj*}{Conjecture}

\newtheorem*{ques*}{Question}

\newtheorem{obs}[thm]{Observation}

\theoremstyle{definition}
\newtheorem{df}[thm]{Definition}
\newtheorem*{df*}{Definition}
\newtheorem*{dfs*}{Definitions}

\newtheorem*{exercise*}{Exercise}

\theoremstyle{remark}
\newtheorem{rem}[thm]{Remark}
\newtheorem*{rem*}{Remark}

\newtheorem*{example*}{Example}

\usepackage{etoolbox}
\patchcmd{\thmhead}{(#3)}{#3}{}{}
\makeatletter
\g@addto@macro\bfseries{\boldmath}
\makeatother

\newcommand{\cl}[1]{\mathcal{#1}}
\newcommand{\fk}[1]{\mathfrak{#1}}

\newcommand{\sr}[1]{\mathscr{#1}}
\newcommand{\Z}{\mathbf{Z}} 
\newcommand{\R}{\mathbf{R}} 
\newcommand{\C}{\mathbf{C}} 

\newcommand{\x}{\times}

\newcommand{\dt}{\mathrm{d}t}

\renewcommand{\sl}{\fk{sl}}

\newcommand{\ol}[1]{\overline{#1}}

\DeclareMathOperator{\Tr}{Tr}

\DeclareMathOperator{\Stab}{Stab}

\DeclareMathOperator{\Tor}{Tor}
\DeclareMathOperator{\pr}{pr}

\DeclareMathOperator{\Id}{Id}

\DeclareMathOperator{\SU}{SU}
\DeclareMathOperator{\U}{U}

\DeclareMathOperator{\BU}{BU}
\DeclareMathOperator{\PU}{PU}
\DeclareMathOperator{\SO}{SO}

\DeclareMathOperator{\Kh}{Kh}

\DeclareMathOperator{\KR}{KR}
\DeclareMathOperator{\KRC}{KRC}

\DeclareMathOperator{\Sym}{Sym}

\DeclareMathOperator{\Gr}{Gr}

\usepackage{graphicx}
\usepackage{subcaption}
\usepackage{mathtools}
\usepackage{xcolor}
\usepackage{colortbl}
\usepackage{makecell}
\usepackage{wrapfig}
\usetikzlibrary{calc,fit,backgrounds}
\usetikzlibrary{decorations.markings}
\tikzset{
	neg/.style={postaction=decorate,
	decoration={markings,
	mark=at position 0.05cm with \node{$\circ$};
	}}
}
\usepackage{pinlabel}
\usepackage{nicematrix}

\title{Link homology and loop homology}
\author{Joshua Wang}
\date{}

\begin{document}
\maketitle

\begin{abstract}
	We compute the $k$-colored $\sl(N)$ homology of the torus knot $T(2,2m+1)$, and we show that it stabilizes as $m\to\infty$ to the integral homology of the free loop space of the complex Grassmannian $\Gr(k,N)$.  In particular, when $k = 1$ and $N = 2$, we observe that the Khovanov homology of $T(2,2m+1)$ stabilizes to the homology of the free loop space of the $2$-sphere. 
\end{abstract}

\section{Introduction}\label{sec:introduction}

The purpose of this paper is to present a novel connection between link homology and loop homology, afforded by new computations of link invariants. We first explain the simplest instance of the connection, for which new computations are not needed. Consider the Khovanov homology groups \cite{MR1740682} of the torus knot $T(2,2m+1)$, shown in Figure~\ref{fig:KhovanovT2m}. After a downward grading shift by $2m-1$, they stabilize as $m \to \infty$ to a bigraded abelian group known as the stable Khovanov homology of $T(2,\infty)$, shown in Figure~\ref{fig:StableKhT2infty}.

\begin{figure}[!ht]
	\centering
	\vspace{-10pt}\[
		\def\arraystretch{1}
		\scalebox{1.0}{
		\begin{tabular}{ c!{\vrule width 0.9pt}c|c|c|c|c|c| }
		\hline
		$\!15\!$ & $\phantom \Z$ & & & & & \\
		\hline
		$\!13\!$ & & $\!\:\phantom{\Z_2}\!\!$ & & & & \\
		\hline
		$\!11\!$ & & $\phantom \Z$ & $\phantom \Z$ & & & \\
		\hline
		$9$ & & & & $\!\:\phantom{\Z_2}\!\!$ & & \\
		\hline
		$7$ & & & & $\phantom \Z$ & $\phantom \Z$ & \\
		\hline
		$5$ & & & & & & $\!\:\phantom{\Z_2}\!\!$ \\
		\hline
		$3$ & & & & & & $\phantom \Z$\\
		\hline
		$1$ & & & & & & $\Z$\\
		\hline
		$\!\!-1\!\!$ & & & & & & $\Z$\\
		\Xhline{0.9pt}
		& $\!-5\!$ & $\!-4\!$ & $\!-3\!$ & $\!-2\!$ & $\!-1\!$ & $0$
		\end{tabular}
		} \hspace{5pt} \scalebox{1.0}{
		\begin{tabular}{ c!{\vrule width 0.9pt}c|c|c|c|c|c|c|c|c|c|c }
		\hline
		$\!15\!$ & & $\!\:\phantom{\Z_2}\!\!$ & & & & \\
		\hline
		$\!13\!$ & & $\phantom \Z$ & $\phantom \Z$ & & &\\
		\hline
		$\!11\!$ & & & & $\!\:\phantom{\Z_2}\!\!$ & & \\
		\hline
		$9$ & & $\phantom \Z$ & $\Z$ & & & \\
		\hline
		$7$ & & & & $\!\:\Z_2\!\!$ & & \\
		\hline
		$5$ & & & & $\Z$ & &\\
		\hline
		$3$ & & & & & & $\Z$\\
		\hline
		$1$ & & & & & & $\Z$\\
		\hline
		$\!\!-1\!\!$ & $\phantom \Z$ & & & & & \\
		\Xhline{0.9pt}
		& $\!-5\!$ & $\!-4\!$ & $\!-3\!$ & $\!-2\!$ & $\!-1\!$ & $0$
		\end{tabular}
		} \hspace{5pt} \scalebox{1.0}{
		\begin{tabular}{ c!{\vrule width 0.9pt}c|c|c|c|c|c|c|c|c|c|c }
		\hline
		$\!15\!$ & $\Z$ & & & & & \\
		\hline
		$\!13\!$ & & $\!\:\Z_2\!\!$ & & & & \\
		\hline
		$\!11\!$ & & $\Z$ & $\Z$ & & & \\
		\hline
		$9$ & & & & $\!\:\Z_2\!\!$ & & \\
		\hline
		$7$ & & & & $\Z$ & &\\
		\hline
		$5$ & & & & & & $\Z$\\
		\hline
		$3$ & & & & & & $\Z$\\
		\hline
		$1$ & & & & & & \\
		\hline
		$\!\!-1\!\!$ & & & & & & \\
		\Xhline{0.9pt}
		& $\!-5\!$ & $\!-4\!$ & $\!-3\!$ & $\!-2\!$ & $\!-1\!$ & $0$
		\end{tabular}
		}\vspace{-10pt}
	\]
		\captionsetup{width=.8\linewidth}
		\caption{The Khovanov homology groups of the unknot $T(2,1)$, the trefoil $T(2,3)$, and the cinquefoil $T(2,5)$. We use the Khovanov--Rozansky $\sl(2)$ link homology grading conventions \cite{MR2391017}, with the $q$-grading vertical and the $t$-grading horizontal.}
		\label{fig:KhovanovT2m}
\end{figure}

\begin{wrapfigure}{r}{0.45\textwidth}
	\centering
	\vspace{-25pt}\[
		\def\arraystretch{1.00}
		\scalebox{1.0}{
		\begin{tabular}{ c!{\vrule width 0.9pt}c|c|c|c|c|c|c|c|c| }
		\hline
		$\!16\!$ & \!\raisebox{-2pt}{$\smash{\ddots}$}\! & & & & & & & \\
		\hline
		$\!14\!$ & \!\raisebox{-2pt}{$\smash{\ddots}$}\! & $\!\:\Z_2\!\!$ & & & & & &\\
		\hline
		$\!12\!$ & & $\Z$ & $\Z$ & & & & & \\
		\hline
		$\!10\!$ & & & & $\!\:\Z_2\!\!$ & & & & \\
		\hline
		$8$ & & & & $\Z$ & $\Z$ & & & \\
		\hline
		$6$ & & & & & & $\!\:\Z_2\!\!$ & & \\
		\hline
		$4$ & & & & & & $\Z$ & & \\
		\hline
		$2$ & & & & & & & & $\Z$\\
		\hline
		$0$ & & & & & & & & $\Z$\\
		\Xhline{0.9pt}
		& $\!-7\!$ & $\!-6\!$ & $\!-5\!$ & $\!-4\!$ & $\!-3\!$ & $\!-2\!$ & $\!-1\!$ & $0$
		\end{tabular}
		}\vspace{-15pt}
	\]
	\captionsetup{width=.8\linewidth}
	\caption{The stable Khovanov homology of $T(2,\infty)$.}
	\label{fig:StableKhT2infty}
	\vspace{-10pt}
\end{wrapfigure}

The Khovanov homology groups of the torus link $T(n,m)$, with a suitable $q$-grading shift, also stabilize as $m\to\infty$ to a limit known as the stable Khovanov homology of $T(n,\infty)$ \cite{MR2492301}. More precisely, there are chain-level maps defined between the relevant complexes underlying these homology groups, and the stable limit is a colimit at the chain level. The stable Khovanov homology of $T(n,\infty)$ is also the invariant of the unknot in the Cooper--Krushkal categorification of the $n$-colored Jones polynomial \cite{MR2901969}, and these stable groups are the subject of the Gorsky--Oblomkov--Rasmussen conjecture \cite{MR3171092}. Although they play an important role in the categorified representation theory of the quantum group $U_q(\sl(2))$, they have only been computed explicitly for $n \in \{2,3\}$. 

\begin{wrapfigure}{r}{0.32\textwidth}
	\centering
	\vspace{-10pt}\[
		\begin{tabular}{ c!{\vrule width 0.9pt}c| }
		\hline
		$2$ & $\Z$\\
		\hline
		$0$ & $\Z$\\
		\Xhline{0.9pt}
		& $0$
		\end{tabular} \hspace{20pt} \begin{tabular}{ c!{\vrule width 0.9pt}c|c| }
		\hline
		$4$ & $\Z$ &\\
		\hline
		$2$ & & $\!\:\Z_2\!\!$\\
		\hline
		$0$ & & $\Z$\\
		\Xhline{0.9pt}
		& $\!-1\!$ & $0$
		\end{tabular} \vspace{-10pt}
	\]
	\captionsetup{width=.9\linewidth}
	\caption{Two types of summands.}
	\label{fig:TwoTypesOfSummands}
	\vspace{-10pt}
\end{wrapfigure}

Let us examine the stable Khovanov homology of $T(2,\infty)$ more carefully. Up to grading shifts, only two types of direct summands appear, shown in Figure~\ref{fig:TwoTypesOfSummands}. The first type of summand only appears once, while the second type of summand appears infinitely many times with grading shifts $t^{-2l}q^{4l}$ for $l > 0$. Next, for either summand, consider collapsing the bigrading to a single grading by setting $t = q$. The first summand yields the cohomology of $S^2$, while the second summand yields the cohomology of $\SO(3)$. Identifying these summands with the cohomology groups of $S^2$ and $\SO(3)$ goes back to observations of Kronheimer--Mrowka \cite{MR2860345} and Jacobsson--Rubinsztein \cite{https://doi.org/10.48550/arxiv.0806.2902} in connection with meridian-traceless $\SU(2)$ representation varieties of knot groups \cite{MR1158339}. With these interpretations of the summands in mind, we tentatively observe that the stable Khovanov homology of $T(2,\infty)$ is isomorphic to the cohomology groups of $S^2 \sqcup \SO(3) \sqcup \SO(3) \sqcup \cdots$ after certain grading shifts and collapses. 

The key new observation is that the space $S^2 \sqcup \SO(3) \sqcup \SO(3) \sqcup \cdots$ arises in a different context; it is precisely the space of closed geodesics on the round $2$-sphere.
The copy of $S^2$ is the space of constant loops, and the\begin{wrapfigure}{r}{0.3\textwidth}
	\centering
	\vspace{-5pt}
	\includegraphics[width=.2\textwidth]{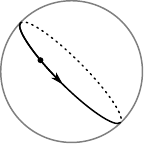}
	\captionsetup{width=.9\linewidth}
	\caption{A closed geodesic.}
	\label{fig:geodesicOnS2}
	\vspace{-5pt}
\end{wrapfigure} $l$th copy of $\SO(3)$ is the space of closed geodesics that go $l$ times around a great equator. For us, a closed geodesic is a map from $S^1$ to $S^2$ that satisfies the geodesic equation. We view the space of closed geodesics as a subspace of the free loop space of $S^2$, denoted by $LS^2$. It is the critical set of the energy functional $E\colon LS^2 \to \R$ given by $E(\gamma) = \frac12\int_{S^1} |\dot{\gamma}|^2 \,\dt$ which, for this metric, is a perfect Morse--Bott function \cite{MR649625}. In particular, the homology of $LS^2$ is isomorphic to the direct sum of the homology groups of the connected components of the space of closed geodesics, shifted by Morse index: \vspace{-5pt}\[
	H_*(LS^2) \cong H_*(S^2) \oplus \bigoplus_{l=1}^\infty H_{*-(2l-1)}(\SO(3)).\vspace{-5pt}
\]

\begin{obs}\label{obs:KhT2infandHLS2}
	The stable Khovanov homology of $T(2,\infty)$ coincides with the homology of the free loop space of $S^2$. 
\end{obs}

At the most basic level, the observation is simply that both sides are isomorphic to the homology or cohomology of $S^2 \sqcup \SO(3) \sqcup \SO(3) \sqcup \cdots$ with suitable shifts. To compare the gradings on the two sides more precisely, we first set $\mathbf{H}_*(LB) \coloneqq H_{*\,+\,\dim B}(LB)$ and $\mathbf{H}_*(B) \coloneqq H_{*\,+\,\dim B}(B)$ where $B$ is a connected (finite-dimensional) manifold. The first convention is standard in string topology because the Chas--Sullivan loop product \cite{chas1999stringtopology} takes the form of a graded-commutative map $\mathbf{H}_*(LB) \x \mathbf{H}_*(LB) \to \mathbf{H}_*(LB)$ with this grading shift. We use the grading shift notation $s^i\,\mathbf{H}_* \coloneqq \mathbf{H}_{*-i}$. With these conventions, we have \vspace{-5pt} \begin{equation}\label{eq:LS2gradings}
	\mathbf{H}_*(LS^2) \cong \mathbf{H}_*(S^2) \oplus \bigoplus_{l=1}^\infty s^{2l}\mathbf{H}_*(\SO(3)).\vspace{-5pt}
\end{equation}
Next, we explain how the bigradings on $H^*(S^2)$ and $H^*(\SO(3))$ in Figure~\ref{fig:TwoTypesOfSummands} actually arise from the standard $\U(2)$-actions on $S^2$ and $\SO(3)$. Each is a homogeneous space $\U(2)/J$ for a closed subgroup $J \subseteq \U(2)$. Its Borel equivariant cohomology $H^*_{\hspace{1pt}\smash{\U(2)}}(\U(2)/J)$ may be identified with $H^*(\mathrm{B}J)$ as a module over $H^*(\BU(2))$. Its ordinary cohomology $H^*(\U(2)/J)$ is not obtained from its equivariant cohomology by naively killing the action of $H^*(\BU(2))$ by forming the tensor product $H^*(\mathrm{B}J) \otimes_{H^*(\BU(2))} \Z$; instead, it is isomorphic to the torsion product $\Tor_{H^*(\BU(2))}(H^*(\mathrm{B}J),\Z)$ by a theorem of Gugenheim--May \cite{MR0394720}. More precisely, the singly-graded group $H^*(\U(2)/J)$ is obtained by collapsing the bigrading by $t = q$ on the torsion product, where $q$ is the internal grading on the modules and $t$ is the cohomological grading. We set $\mathbf{H}^{**}(\U(2)/J) \coloneqq \Tor_{H^*(\BU(2))}(H^*(\mathrm{B}J),\Z)$ and $t^iq^j\,\mathbf{H}^{**} \coloneqq \mathbf{H}^{(*-i)(*-j)}$. The bigraded groups $\mathbf{H}^{**}(S^2)$ and $\mathbf{H}^{**}(\SO(3))$ are precisely the two types of summands in Figure~\ref{fig:TwoTypesOfSummands}, so the stable Khovanov homology of $T(2,\infty)$ is \vspace{-5pt}\begin{equation}\label{eq:KhT2infgradings}
	\Kh(T(2,\infty)) \cong \mathbf{H}^{**}(S^2) \oplus \bigoplus_{l=1}^\infty t^{-2l}q^{4l} \mathbf{H}^{**}(\SO(3)).\vspace{-5pt}
\end{equation}
Equations~(\ref{eq:LS2gradings}) and (\ref{eq:KhT2infgradings}) provide a grading-refinement of Observation~\ref{obs:KhT2infandHLS2}. Also see Remark~\ref{rem:cyclicGrading} below.

The main result of this paper establishes the same relationship for the stable $k$-colored $\sl(N)$ homology of $T(2,\infty)$ and the homology of the free loop space of the complex Grassmannian $\Gr(k,N)$. The main work that enables this comparison is a complete computation of the $k$-colored $\sl(N)$ homology of $T(2,2m+1)$, which is new beyond the case of the trefoil \cite{MR4903257}. These richer knot invariants are challenging to compute, and our computations here are the main application of \cite{wang2025minimalrickardcomplexesbraids}.

To state our results, we summarize some facts about closed geodesics on $\Gr(k,N)$ equipped with the standard metric that makes it a Riemannian symmetric space. See section~\ref{sec:loop_homology_of_the_grassmannian} for these results, which are derived from \cite{MR649625}. \begin{itemize}[noitemsep]
	\item[$\bullet$] The connected components of the space of closed geodesics on $\Gr(k,N)$ may be indexed by the set of partitions with at most $\min(k,N-k)$ parts. We let $B_\lambda$ denote the connected component corresponding to $\lambda = (\lambda_1,\ldots,\lambda_r)$. Set $|\lambda| = \lambda_1 + \cdots + \lambda_r$, $|\lambda|_2 = (\lambda_1^2 + \cdots + \lambda_r^2)^{1/2}$, and $|\lambda|_\infty = \max(\lambda_1,\ldots,\lambda_r) = \lambda_1$.
	\item[$\bullet$] The space $B_\lambda$ is a homogeneous space $\U(N)/J_\lambda$ for a closed connected subgroup $J_\lambda \subseteq \U(N)$. The subgroup $J_\lambda$ depends only on the multiplicities of the parts of $\lambda$ and is given explicitly in Proposition~\ref{prop:BlambdaAsUnitaryOrbit}. Set $\mathbf{H}^{**}(B_\lambda) \coloneqq \Tor_{H^*(\BU(N))}(H^*(\mathrm{B}J_\lambda),\Z)$. The cohomology groups of $B_\lambda$ are obtained from $\mathbf{H}^{**}(B_\lambda)$ by collapsing the bigrading by $t = q$ \cite{MR0394720}, and $\mathbf{H}_*(B_\lambda)$ is further obtained by negating the resulting grading. 
	\item[$\bullet$] The energy functional $E\colon L\Gr(k,N) \to \R$ is a perfect Morse--Bott function. For any closed geodesic in $B_\lambda$, its length is $\pi|\lambda|_2$, its energy is $\pi^2|\lambda|_2^2/2$, and its Morse index and nullity sum to $2k(N-k) + (2N+2)|\lambda| - 4\sum_{i=1}^r i\lambda_i$.
\end{itemize}

\begin{thm}\label{thm:mainTheorem}
	For $0 < k < N$, the stable $k$-colored $\sl(N)$ homology of $T(2,\infty)$ coincides with the homology of the free loop space of the complex Grassmannian $\Gr(k,N)$. More precisely, we have isomorphisms of (bi)graded abelian groups \begin{align*}
		\mathbf{H}_*(L\Gr(k,N)) &\cong \bigoplus_{\lambda} s^{(2N+2)|\lambda| - 4\sum_i i\lambda_i} \mathbf{H}_*(B_\lambda)\\
		\KR_{k,N}(T(2,\infty)) &\cong \bigoplus_\lambda t^{-2|\lambda|}q^{4\sum_i i\lambda_i} \mathbf{H}^{**}(B_\lambda)
	\end{align*}
	where both direct sums are over partitions with at most $\min(k,N-k)$ parts. The invariant of $T(2,2m+1)$ is \[
		\KR_{k,N}(T(2,2m+1)) \cong q^{k(N-k)(2m-1)} \bigoplus_{|\lambda|_\infty \leq m} t^{-2|\lambda|}q^{4\sum_i i\lambda_i} \mathbf{H}^{**}(B_\lambda)
	\]
	where the direct sum is over partitions with at most $\min(k,N-k)$ parts, each of size at most $m$.
\end{thm}

\begin{rem}\label{rem:cyclicGrading}
	By collapsing and negating the bigrading on $\KR_{k,N}(T(2,\infty))$, we obtain $\bigoplus_{\lambda} s^{2|\lambda| - 4\sum_i i\lambda_i} \mathbf{H}_*(B_\lambda)$ which differs from $\mathbf{H}_*(L\Gr(k,N))$ by a grading shift of $s^{2N|\lambda|}$ on the summand indexed by $\lambda$. So the graded groups agree if the $\Z$-grading is reduced to a cyclic $\Z/2N$-grading, but the groups are no longer finitely generated in each grading. 
\end{rem}

This coincidence between link homology and loop homology is mysterious to the author. The proof of Theorem~\ref{thm:mainTheorem} is entirely computational, but we believe that the result demands a conceptual explanation. The appearance of the number $2N$ in relation to the grading suggests that Floer theory may play a role. A cyclic grading by $\Z/2N$ arises in instanton Floer homology with gauge group $\SU(N)$ \cite{MR2860345} and in Lagrangian intersection homology within a monotone symplectic manifold with minimal Chern number $N$ \cite{MR987770,MR1223659}. For lack of additional evidence, we refrain from further speculation here. 

We note that Observation~\ref{obs:KhT2infandHLS2} is distinct from the connection between the Khovanov homology of $T(2,2m)$ and the Hochschild homology of $\Z[x]/(x^2)$ of \cite{rozansky2010categorificationstablesu2wittenreshetikhinturaev,MR2657644}. Their connection is most naturally phrased in terms of a stable limit of the Khovanov homology groups of the torus links $T(2,2m)$ with oppositely oriented strands that corresponds to an antisymmetric projector rather than the symmetric projector considered here. 

In section~\ref{sec:loop_homology_of_the_grassmannian}, we review free loop spaces of compact symmetric spaces following Ziller \cite{MR649625}, specialized to the Grassmannian. We upgrade his main result from $\Z/2$-coefficients to $\Z$-coefficients for the Grassmannian and prove the first part of Theorem~\ref{thm:mainTheorem}. In section~\ref{sec:colored_link_homology_of_two_stranded_torus_knots}, we apply \cite{MR4903257,wang2025minimalrickardcomplexesbraids} to compute the $k$-colored $\sl(N)$ homology of $T(2,\infty)$ and $T(2,2m+1)$ and prove the second part of Theorem~\ref{thm:mainTheorem}. The proof uses a chain action of the nil-Hecke algebra $\cl{H}_r$ on an auxiliary complex $C^r$. In Appendix~\ref{sec:nil_hecke_action}, we verify the nil-Hecke relations of this action by straightforward but lengthy computations. 

\theoremstyle{definition}
\newtheorem*{ack}{Acknowledgements}
\begin{ack}
	I would like to thank William Ballinger, Elijah Bodish, Luke Conners, Eugene Gorsky, Matthew Hogancamp, Mikhail Khovanov, Peter Kronheimer, Tomasz Mrowka, Rohil Prasad, David Rose, Raphael Rouquier, Lev Rozansky, Matthew Stoffregen, Catharina Stroppel, Joshua Sussan, Emmanuel Wagner, Paul Wedrich, Alex Waldron, and Michael Willis for many helpful and inspiring conversations. This work is partially supported by the NSF MSPRF grant DMS-2303401 and the Simons Collaboration on ``New Structures in Low-Dimensional Topology''.
\end{ack}

\section{Loop homology of the Grassmannian}\label{sec:loop_homology_of_the_grassmannian}

We review the basics of closed geodesics in a compact symmetric space, specialized to the case of the complex Grassmannian, following Ziller \cite{MR649625}. Our presentation is more explicit than what appears in the literature for two reasons. First, it allows us to upgrade Ziller's main result for the Grassmannian from $\Z/2$-coefficients to $\Z$-coefficients. Second, it allows us to make a comparison with link homology. 

Fix integers $k$ and $N$ that satisfy $0 < k < N$, and let $M \coloneqq \Gr(k,N)$ be the Grassmannian of $k$-dimensional vector subspaces of $\C^N$. For notational simplicity, we assume that $k \leq N-k$ by passing from $\Gr(k,N)$ to $\Gr(N-k,N)$ by taking orthogonal complements. Let $\Lambda_{\text{std}} \in M$ be the span of the first $k$ standard basis vectors of $\C^N$. The special unitary group $G \coloneqq \SU(N)$ acts transitively on $M$, and the stabilizer of $\Lambda_{\text{std}}$ is $K \coloneqq \mathrm{S}(\U(k) \x \U(N-k))$. Let $J \coloneqq -\Id_k \oplus \Id_{N-k} \in \U(N)$, and let $\sigma\colon G \to G$ be the involutive automorphism given by conjugation by $J$. Then $K$ is the fixed point set of $\sigma$, so $(G,K)$ is a symmetric pair and $M = G/K$ is a Riemannian symmetric space. 

Let $\fk{g} \coloneqq \fk{su}(N)$ be the Lie algebra of $G$, and let $\fk{k} \subseteq \fk{g}$ be the Lie algebra of $K$. Let $\fk{m} \subseteq \fk{g}$ be the orthogonal complement to $\fk{k}$ with respect to the Killing form. Explicitly, we have \[
	\fk{k} = \left\{ \begin{pmatrix}
		X & 0\\
		0 & Y
	\end{pmatrix} \:\Bigg|\: X\in \fk{u}(k), Y \in \fk{u}(N-k), \Tr(X) + \Tr(Y) = 0 \right\} \qquad \fk{m} = \left\{ \begin{pmatrix}
		0 & -Z^*\\
		Z & 0
	\end{pmatrix} \:\Bigg|\: Z \in \C^{(N-k) \x k} \right\}
\]
where $Z^*$ denotes the conjugate transpose of $Z$. The subspaces $\fk{k}$ and $\fk{m}$ of $\fk{g}$ are the $\pm1$ eigenspaces of $(D\sigma)_{\Id}$. Let $\pr\colon G \to M$ denote the projection map given by $g \mapsto g\Lambda_{\text{std}}$. We identify $\fk{m}$ with the tangent space of $M$ at $\Lambda_{\text{std}}$ by restricting $(D\pr)_{\Id}$ to $\fk{m}$. Under this identification, the Riemannian exponential map is given by $\pr\circ\exp$ where $\exp\colon \fk{m} \to G$ is the matrix exponential map.

We define an explicit basis for $\fk{m}$. Note that a matrix in $\fk{m}$ is determined by its bottom left block of size $(N-k) \x k$. For $i \in \{1,\ldots,N-k\}$ and $j \in \{1,\ldots,k\}$, let $E_{ij} \in \fk{m}$ denote the matrix whose bottom left block has $(i,j)$ entry equal to $1$ and all other entries zero. Let $F_{ij} \in \fk{m}$ denote the matrix whose bottom left block has $(i,j)$ entry equal to $\sqrt{-1}$ and all other entries zero. Note that the two nonzero entries of $E_{ij}$ are $1$ and $-1$, while the two nonzero entries of $F_{ij}$ are both $\sqrt{-1}$. The $2k(N-k)$ matrices $E_{ij},F_{ij}$ form a basis for $\fk{m}$.

Let $\fk{a} \subseteq \fk{m}$ be the $k$-dimensional subspace with basis $E_{ii}$ for $i \in \{1,\ldots,k\}$. So \[
	\fk{a} = \left\{ \begin{pmatrix}
		0 & -A^t\\
		A & 0
	\end{pmatrix} \:\Bigg|\: A \in \R^{(N-k) \x k} \text{ is diagonal } \right\}.
\]
The image of $\fk{a}$ under the Riemannian exponential is a flat totally geodesic torus $A = (S^1)^k$ within $M$. For $i \in \{1,\ldots,k\}$, the $i$th coordinate factor of $A$ is $\pr(\exp(tE_{ii}))$ for $t \in [0,\pi]$.

We have defined the relevant vector subspaces: $\fk{a} \subseteq \fk{m}$ and $\fk{g} = \fk{k} \oplus \fk{m}$. Next, we consider adjoint actions. Note that $\fk{m}$ is not a Lie subalgebra of $\fk{g}$. Instead, we have $[\fk{m},\fk{m}] \subseteq \fk{k}$ and $[\fk{m},\fk{k}] \subseteq \fk{m}$. The subspace $\fk{a} \subseteq \fk{m}$ is a maximal abelian subspace, where abelian means $[\fk{a},\fk{a}] = 0$. For $x \in \fk{a}$, consider the linear map $\fk{m} \to \fk{m}$ given by $y \mapsto [x,[x,y]]$. The endomorphisms of $\fk{m}$ assigned to any pair of elements of $\fk{a}$ commute, so we may attempt to simultaneously diagonalize the entire action of $\fk{a}$ on $\fk{m}$. This may be done successfully, and below is the result. In each case, we specify a linear map $\alpha\colon \fk{a} \to \R$ and a subspace $\fk{m}_\alpha\subseteq \fk{m}$ such that $[x,[x,y]] = -\pi^2\alpha(x)^2 y$ for all $y\in\fk{m}_\alpha$ and $x \in \fk{a}$. \begin{itemize}
	\item[$\bullet$] For $j \in \{1,\ldots,k\}$, let $\fk{m}_{\alpha_j} \coloneqq \mathrm{span}(F_{jj})$ and $\alpha_j(E_{ii}) = \begin{cases}
					2/\pi & i = j\\
					0 & i \neq j.
				\end{cases}$
	\item[$\bullet$] For $j \in \{1,\ldots,k\}$, let $\fk{m}_{\beta_j} \coloneqq \mathrm{span}(E_{ij},F_{ij} \:|\: i \in \{k+1,\ldots,N-k\})$ and $\beta_j(E_{ii}) = \begin{cases}
		1/\pi & i = j\\
		0 & i \neq j.
	\end{cases}$
	\item[$\bullet$] For $i,j \in \{1,\ldots,k\}$ with $i < j$, let $\fk{m}_{\gamma_{ij}}\coloneqq \mathrm{span}(E_{ij}-E_{ji},F_{ij}+F_{ji})$ and $\gamma_{ij}(E_{ll}) = \begin{cases}
			1/\pi & l \in \{i,j\}\\
			0 & l \notin \{i,j\}.
		\end{cases}$
	\item[$\bullet$] For $i,j \in \{1,\ldots,k\}$ with $i < j$, let $\fk{m}_{\delta_{ij}}\coloneqq \mathrm{span}(E_{ij}+E_{ji},F_{ij} - F_{ji})$ and $\delta_{ij}(E_{ll}) = \begin{cases}
		1/\pi & l = i\\
		-1/\pi & l = j\\
		0 & l \notin \{i,j\}.
	\end{cases}$
\end{itemize}
We obtain a direct sum decomposition \[
	\fk{m} = \fk{a} \oplus \bigoplus_j \fk{m}_{\alpha_j} \oplus \bigoplus_j \fk{m}_{\beta_j} \oplus \bigoplus_{i<j} \fk{m}_{\gamma_{ij}} \oplus \bigoplus_{i<j} \fk{m}_{\delta_{ij}}.
\]
Let $R_+ \coloneqq \{\alpha_j,\beta_j\}_{j} \sqcup \{\gamma_{ij},\delta_{ij}\}_{i < j}$. Regarding terminology, the elements of $R_+$ are essentially the positive roots, except that $\beta_j$ is not considered a positive root when $N = 2k$ because $\fk{m}_{\beta_j} = 0$. In order to avoid handling the case $N = 2k$ separately, we include $\beta_j$ within $R_+$ but refrain from referring to the elements of $R_+$ as positive roots.

We now direct our attention within $\fk{a}$. Each $\alpha \in R_+$ determines a hyperplane $\{\alpha = 0\} \subseteq \fk{a}$. Consider the union of these hyperplanes over all $\alpha \in R_+$. The connected components of the complement of this union are called the Weyl chambers of $\fk{a}$. We define the positive Weyl chamber $F$ to be \[
	F \coloneqq \left\{\:x\in\fk{a} \:\big|\: \alpha(x) > 0 \text{ for all }\alpha\in R_+\:\right\} = \left\{\:a_{1}E_{11} + \cdots + a_{k}E_{kk} \in \fk{a} \:\big|\: a_{1} > \cdots > a_{k} > 0\:\right\},
\]and we note that its closure is $\ol{F} = \{\:a_{1}E_{11} + \cdots + a_{k}E_{kk} \in \fk{a} \:|\: a_{1} \ge \cdots \ge a_{k} \ge 0\:\}$. Next, let $L \subseteq \fk{a}$ be the lattice consisting of all $\Z$-linear combinations of $\pi E_{11},\ldots,\pi E_{kk}$. A matrix $x \in \fk{a}$ satisfies $\pr(\exp(x)) = \Lambda_{\text{std}}$ if and only if $x \in L$. There is a bijection between $L$ and the set of closed geodesics within the flat totally geodesic torus $A \subseteq M$ based at $\Lambda_{\text{std}} = \pr(\exp(0))$, given by sending $x\in L$ to the closed geodesic $t\mapsto \pr(\exp(tx))$ for $t \in [0,1]$. The intersection $L \cap \ol{F}$ consists of the points \[
	L \cap \ol{F} = \left\{\: \lambda_1\pi E_{11} + \cdots + \lambda_k \pi E_{kk} \in \fk{a} \:\big|\: (\lambda_1,\ldots,\lambda_k) \in\Z^k \:\text{ and } \lambda_1 \ge \cdots \ge \lambda_k \ge 0 \: \right\}
\]So there is a bijection between $L \cap \ol{F}$ and the set of partitions $\lambda$ with at most $k$ parts. See \cite[section 1.3]{MR649625} for the following proposition.

\begin{prop}\label{prop:partitionsAndCriticalSubmanifolds}
	There is a bijection between partitions with at most $k$ parts and connected components of the space of closed geodesics on $M$. The bijection $\lambda \mapsto B_\lambda$ is defined in the following way. Set $x_\lambda \coloneqq \lambda_1\pi E_{11} + \cdots + \lambda_k \pi E_{kk} \in L \cap \ol{F} \subseteq \fk{a}$, and let $\gamma_\lambda$ be the closed geodesic given by $\gamma_\lambda(t) \coloneqq \pr(\exp(tx_\lambda))$. Define $B_\lambda\coloneqq\{\,g\gamma_\lambda \:|\: g\in G\}$ to be the $G$-orbit of $\gamma_\lambda$.
\end{prop}

Let $LM$ be the free loop space of $M$. The energy functional $E\colon LM \to \R$ is given by $E(\gamma) = \frac12\int_{S^1} \|\dot{\gamma}\|^2 \:\dt$. The critical points of $E$ are precisely the closed geodesics on $M$. The critical points are not isolated; instead, they come in families of closed submanifolds. These critical submanifolds are the connected components $B_\lambda$ of the space of closed geodesics, indexed by partitions $\lambda$ with at most $k$ parts. The critical submanifolds are all nondegenerate by \cite[Theorem 2]{MR649625}. Let $\gamma_\lambda \in B_\lambda$ be the closed geodesic $\gamma_\lambda(t) = \pr(\exp(tx_\lambda))$ where $x_\lambda = \sum_{i=1}^k \lambda_i \pi E_{ii}$. Its length is $\ell(\gamma_\lambda) = \pi|\lambda|_2$, its energy is $E(\gamma_\lambda) = \pi^2|\lambda|_2^2/2$, and its Morse index $\mu$ and nullity $\nu$ are \begin{align*}
	\mu(\gamma_\lambda) &= \sum_{\substack{\alpha\in R_+\\\alpha(x_\lambda) > 0}} (\alpha(x_\lambda)-1) \cdot \dim \fk{m}_\alpha & \nu(\gamma_\lambda) &= \dim M + \sum_{\substack{\alpha\in R_+\\\alpha(x_\lambda) > 0}} \dim\fk{m}_\alpha
\end{align*}by \cite[Theorem 3]{MR649625}. Nondegeneracy of the critical submanifold $B_\lambda$ implies that $\dim B_\lambda = \nu(\gamma_\lambda)$. The following proposition identifies $B_\lambda$ explicitly.

\begin{prop}\label{prop:BlambdaAsUnitaryOrbit}
	Let $\lambda = (\lambda_1,\ldots,\lambda_k)$ be a partition with at most $k$ parts. Let $r_1,\ldots,r_m$ be the multiplicities of its parts. In other words, they are the unique positive integers for which \[
		\lambda_1 = \cdots = \lambda_{r_1} > \lambda_{r_1+1} = \cdots = \lambda_{r_1+r_2} > \cdots > \lambda_{r_1+\cdots+r_m+1} = \cdots=\lambda_{k}= 0.
	\]Let $r \coloneqq r_1+\cdots+r_m$ be the number of nonzero parts of $\lambda$. Then $B_\lambda$ is the homogeneous space $B_\lambda = \U(N)/J_\lambda$ where \[
		J_\lambda \coloneqq \left\{ \hspace{5pt} \left(\setlength{\tabcolsep}{0.1em}\def\arraystretch{0.8}\begin{tabular}{cccccccc}
			$\smash{U_1}$\\
			& $\smash{\ddots}$\\
			& & $\smash{U_m}$\\
			& & & $\smash{V}$\\
			& & & & $\smash{U_1}$\\
			& & & & & $\smash{\ddots}$\\
			& & & & & & $\smash{U_m}$\\
			& & & & & & & $\smash{W}$
		\end{tabular}\right) \in \U(N) \hspace{5pt} \:\middle|\: \hspace{5pt} \begin{aligned}
			U_1 &\in \U(r_1)\\
			\smash{\vdots}\\[-5pt]
			U_m &\in \U(r_m)\\[-2pt]
			V &\in \U(k-r)\\[-2pt]
			W &\in \U(N-k-r)
		\end{aligned} \hspace{5pt} \right\}.
	\]
\end{prop}
\begin{proof}
	The action of $\SU(N)$ on $B_\lambda$ factors through $\PU(N)$, so we obtain an action of $\U(N)$ on $B_\lambda$ by composing with $\U(N) \to \PU(N)$. We show that the stabilizer of $\gamma_\lambda$ under this unitary action is $J_\lambda$. The stabilizer of $\gamma_\lambda$ certainly stabilizes $\gamma_\lambda(0) = \Lambda_{\text{std}}$ so $\Stab(\gamma_\lambda) \subseteq \U(k) \x \U(N-k) \subseteq \U(N)$. We claim that $g\in \U(k) \x \U(N-k)$ stabilizes $\gamma_\lambda$ if and only if the adjoint action of $g$ on $\fk{g}$ stabilizes $x_\lambda\in \fk{a} \subseteq \fk{g}$. If $g$ stabilizes $\gamma_\lambda$, then differentiating $g\gamma_\lambda(t) = \gamma_\lambda(t)$ at $t = 0$ gives $g\dot{\gamma}_\lambda(0)g^{-1} = \dot{\gamma}_\lambda(0)$. Since $\dot{\gamma}_\lambda(0) = x_\lambda$, we see that the adjoint action of $g$ fixes $x_\lambda$. Conversely, if $gx_\lambda g^{-1} = x_\lambda$, then $g\pr(\exp(tx_\lambda)) = g\pr(\exp(tx_\lambda)g^{-1}) = \pr(\exp(tgx_\lambda g^{-1})) = \pr(\exp(tx_\lambda))$. 

	Assume that $gx_\lambda g^{-1} = x_\lambda$ for $g \in \U(k) \x \U(N-k)$. With $g = S \oplus T$ for $S\in \U(k)$, $T \in \U(N-k)$, we have $TX = XS$ where $X \coloneqq \mathrm{diag}(\lambda_1\pi,\ldots,\lambda_k\pi) \in \R^{(N-k) \x k}$. Next, we express $X,T,S$ as block matrices where the upper left block is size $r \x r$: \[
		X = \begin{pmatrix}
			X_+ & 0\\
			0 & 0
		\end{pmatrix} \qquad S = \begin{pmatrix}
			S_+ & 0\\
			0 & V
		\end{pmatrix} \qquad T = \begin{pmatrix}
			T_+ & 0\\
			0 & W
		\end{pmatrix}.
	\]The off-diagonal blocks of $S$ and $T$ are zero by the identity $TX = XS$ and the fact that $X_+$ is a diagonal matrix with nonzero diagonal entries. It follows that $S_+,T_+ \in \U(r)$ and $V \in \U(k-r)$ and $W \in \U(N-k-r)$. The identity $TX = XS$ gives $T_+X_+ = X_+ S_+$. Invertibility of $X_+$ gives $X_+S_+(X_+)^{-1} = T_+ \in \U(r)$ so \[
		\Id = T_+^* T_+ = (X_+S_+(X_+)^{-1})^*X_+S_+(X_+)^{-1}.
	\]Hence $S_+X_+^2 = X_+^2S_+$. Commuting diagonalizable matrices are simultaneously diagonalizable, so $S_+ = U_1 \oplus \cdots \oplus U_m$ where $U_i \in \U(r_i)$. It follows that $S_+$ and $X_+$ commute so $T_+ = X_+S_+(X_+)^{-1} = S_+$. Thus, $g$ lies in $J_\lambda$. That $g$ stabilizes $x_\lambda$ if $g \in J_\lambda$ is straightforward to verify.
\end{proof}

The following theorem is essentially \cite[Theorem 5]{MR649625}, specialized to the Grassmannian and with coefficients upgraded from $\Z/2$ to $\Z$. As explained by Ziller, upgrading to integer coefficients only requires verifying orientability of certain auxiliary spaces.

\begin{thm}\label{thm:ZillerThm}
	The integral homology of the free loop space of $M = \Gr(k,N)$ is isomorphic to the direct sum \[
		H_*(LM) \cong \bigoplus_{\lambda} H_{* - \mu(\gamma_\lambda)}(B_\lambda)
	\]over partitions with at most $k$ parts.
\end{thm}
\begin{proof}
	Fix a partition $\lambda$ with at most $k$ parts. Based on a construction of Bott--Samelson \cite{MR105694}, Ziller constructs a closed $(\mu(\gamma_\lambda) + \dim B_\lambda)$-dimensional submanifold $\Gamma_\lambda \subseteq LM$, called a ``$K$-cycle'' or a ``completing manifold'', with the property that $B_\lambda \subseteq \Gamma_\lambda \subseteq LM$. The manifold $\Gamma_\lambda$ is a fiber bundle $\Gamma_\lambda \to B_\lambda$ where the inclusion $B_\lambda \subseteq \Gamma_\lambda$ is a section. Furthermore, the normal bundle of $B_\lambda$ within $\Gamma_\lambda$ is isomorphic to the ``negative vector bundle'' $\xi \to B_\lambda$, whose fiber over a point is defined to be the negative eigenspace of the Hessian of the energy functional. So $\xi$ is a rank $\mu(\gamma_\lambda)$ vector bundle over $B_\lambda$. As Ziller notes in \cite[Theorem 5]{MR649625}, it suffices to show that $\xi \to B_\lambda$ is an orientable vector bundle and that $\Gamma_\lambda$ is an orientable manifold. Because the normal bundle of $B_\lambda$ within $\Gamma_\lambda$ is $\xi$, it suffices to show that $B_\lambda$ and $\Gamma_\lambda$ are orientable. Proposition~\ref{prop:BlambdaAsUnitaryOrbit} shows that $B_\lambda = \U(N)/J_\lambda$ is orientable because $J_\lambda$ is connected. 

	To show that $\Gamma_\lambda$ is orientable, we first recall its definition. Let $0 < t_1 < \cdots < t_m < 1$ be the distinct times for which there exists $\alpha \in R_+$ with $\alpha(tx_\lambda) \in \Z_{>0}$. The points $\gamma_\lambda(t_i)\in M$ are precisely the points that are conjugate to $\Lambda_{\text{std}}$ along $\gamma_\lambda$. Let $K_i$ be the stabilizer of $\gamma_\lambda(t_i)$ within $K = \mathrm{S}(\U(k) \x \U(N-k))$. Let $K_\lambda$ be the stabilizer of $\gamma_\lambda$ within $K$ and note that $K_\lambda \subseteq K_i \subseteq K \subseteq G = \SU(N)$. Next, consider the right action of $K_\lambda^{m+1}$ on the product $W \coloneqq G \x K_1 \x \cdots \x K_m$ given by $(g,k_1,\ldots,k_m)\cdot(j_1,\ldots,j_{m+1}) \coloneqq (gj_1,j_1^{-1}k_1j_2,\ldots,j_m^{-1}k_mj_{m+1})$. Define a map $W \to LM$ by sending $(g,k_1,\ldots,k_m)$ to the loop \[
		t \mapsto \begin{cases}
			\begin{tabular}{c|c|c|c|c}
				$g\gamma_\lambda(t)$ & $gk_1\gamma_\lambda(t)$ & $gk_1k_2\gamma_\lambda(t)$ & $\:\:\cdots\:\:$ & $gk_1\cdots k_m \gamma_\lambda(t)$\\
				\hline
				$t \in [0,t_1]$ & $t \in [t_1,t_2]$ & $t \in [t_2,t_3]$ & $\cdots$ & $t\in [t_m,1]$
			\end{tabular}
		\end{cases}.
	\]The map descends to an embedding of the quotient $W/K_\lambda^{m+1} \to LM$, and $\Gamma_\lambda$ is defined to be its image. 

	To see that $\Gamma_\lambda$ is orientable, we simply work with $\U(N)$ instead of $G = \SU(N)$ just as in the proof of Proposition~\ref{prop:BlambdaAsUnitaryOrbit}. Let $V \coloneqq \U(N) \x J_1 \x \cdots \x J_m$ where $J_i \subseteq \U(k) \x \U(N-k)$ is defined to be the stabilizer of $\gamma_\lambda(t_i)$. Then $\Gamma_\lambda$ is diffeomorphic to the quotient of $V$ by the right action of $J_\lambda^{m+1}$ given by the same formula defining the right action of $K_\lambda^{m+1}$ on $W$. Since $J_\lambda$ is connected by Proposition~\ref{prop:BlambdaAsUnitaryOrbit} and because $V$ is orientable, it follows that $\Gamma_\lambda$ is orientable.
\end{proof}

\begin{cor}\label{cor:freeloopspaceshifts}
	We have \[
		\mathbf{H}_*(L\Gr(k,N)) \cong \bigoplus_{\lambda} s^{(2+2N)|\lambda| - 4\sum_i i\lambda_i} \mathbf{H}_*(B_\lambda)
	\]where the direct sum ranges over all partitions $\lambda$ with at most $k$ parts.
\end{cor}
\begin{proof}
	By Theorem~\ref{thm:ZillerThm}, we have $\mathbf{H}_*(LM) \cong \bigoplus_\lambda s^{\mu(\gamma_\lambda)+\nu(\gamma_\lambda)-\dim M} \mathbf{H}_*(B_\lambda)$. Using the formulas from \cite[Theorem 3]{MR649625} stated above before Proposition~\ref{prop:BlambdaAsUnitaryOrbit}, we have \begin{align*}
		\mu(\gamma_\lambda) + \nu(\gamma_\lambda) - \dim M &= \sum_{\alpha\in R_+} \alpha(x_\lambda)\cdot\dim\fk{m}_\alpha\\
		&= \sum_{i=1}^k \alpha_i(x_\lambda) + 2(N-2k)\sum_{i=1}^k \beta_i(x_\lambda) + 2\sum_{i<j} \gamma_{ij}(x_\lambda) + 2\sum_{i<j} \delta_{ij}(x_\lambda)\\
		&= \sum_{i=1}^k \left(2\lambda_i + 2(N-2k) \lambda_i + \sum_{j=i+1}^k 2(\lambda_i + \lambda_j) + 2(\lambda_i - \lambda_j)\right) = \sum_{i=1}^k 2\lambda_i(N+1-2i).\qedhere
	\end{align*}
\end{proof}

\section{Colored link homology of two-stranded torus knots}\label{sec:colored_link_homology_of_two_stranded_torus_knots}

In this section, we compute the $k$-colored $\sl(N)$ homology of $T(2,2m+1)$ and $T(2,\infty)$. We construct a chain complex $\mathrm{F}_\lambda$ of $H^*(\BU(N))$-modules for each partition $\lambda$ with at most $k$ parts in Definition~\ref{def:Flambda}. The complex $\mathrm{F}_\lambda$ is essentially defined as a subcomplex of another complex $C^r$, defined in Definition~\ref{def:complexC}, where $r$ is the number of parts of $\lambda$. In Proposition~\ref{prop:Flambdafreeresolution}, we show that $\mathrm{F}_\lambda$ is a free resolution of $H^*(\mathrm{B}J_\lambda)$. A key part of the proof is an argument that $\mathrm{F}_\lambda$ is a direct summand of $C^r$, which relies on a nil-Hecke action on $C^r$ which we construct in Appendix~\ref{sec:nil_hecke_action}. In Proposition~\ref{prop:slNcomplexT2m}, we identify the colored $\sl(N)$ complexes of $T(2,2m+1)$ and $T(2,\infty)$ with direct sums of the complexes $\mathrm{F}_\lambda \otimes_{H^*(\BU(N))} \Z$ with suitable grading shifts. We use and assume familiarity with the version of colored $\sl(N)$ homology defined by foams \cite{MR3545951,MR4164001}. An exposition of the construction is provided in \cite[section 2]{MR4903257}. The special case of Proposition~\ref{prop:slNcomplexT2m} for the trefoil $T(2,3)$ is essentially the main result of \cite{MR4903257}, which used the simplified Rickard complex for the full twist from \cite{hogancamp2021skein}. The new input needed to prove Proposition~\ref{prop:slNcomplexT2m} for $T(2,2m+1)$ is the simplified Rickard complex for an arbitrary power of the full twist, constructed in \cite{wang2025minimalrickardcomplexesbraids}. 

Recall that the nil-Hecke algebra $\cl{H}_r$ is the endomorphism ring of the polynomial ring $\Z[x_1,\ldots,x_r]$, viewed as a module over the ring of symmetric polynomials $\Z[x_1,\ldots,x_r]^{\fk{S}_r}$. As an algebra, it is generated by $x_1,\ldots,x_r,\partial_1,\ldots,\partial_{r-1}$ where $\partial_i$ is the \textit{divided difference operator} or \textit{Demazure operator} given by $\partial_i = (\Id - \,s_i)/(x_i - x_{i+1})$ where $s_i$ is the simple transposition swapping $x_i$ and $x_{i+1}$. Note that $s_i = \Id - \,(x_i - x_{i+1})\partial_i = [\partial_i,x_i] = -[\partial_i,x_{i+1}] \in \cl{H}_r$. 

Next, if $\mathbf{X} = \{x_1,\ldots,x_r\}$ is a collection of indeterminates, we use the notational shorthand $\Z[\mathbf{X}] \coloneqq \Z[x_1,\ldots,x_r]$ and $\Sym(\mathbf{X}) \coloneqq \Z[x_1,\ldots,x_r]^{\fk{S}_r}$. Also let $e_i(\mathbf{X}),h_i(\mathbf{X}) \in \Sym(\mathbf{X})$ denote the elementary and complete homogeneous symmetric polynomials, respectively. If $\mathbf{Y} = \{y_1,\ldots,y_m\}$ is another collection of indeterminates, we set \begin{align*}
	e_i(\mathbf{X} - \mathbf{Y}) &\coloneqq \sum_{j=0}^i (-1)^j e_{i-j}(\mathbf{X})h_j(\mathbf{Y}) & h_i(\mathbf{X} - \mathbf{Y}) &\coloneqq \sum_{j=0}^i (-1)^j h_{i-j}(\mathbf{X})e_j(\mathbf{Y}).
\end{align*}
We note that $e_i(\mathbf{X} - \mathbf{Y}) = (-1)^i h_i(\mathbf{Y} - \mathbf{X}) \in \Sym(\mathbf{X}) \otimes \Sym(\mathbf{Y})$. For more details, see for example \cite{hogancamp2021skein,wang2025minimalrickardcomplexesbraids}.

\begin{df}\label{def:complexC}
	Fix $r \ge 0$, and let $\mathbf{X} = \{x_1,\ldots,x_r\}$ and $\mathbf{Y} = \{y_1,\ldots,y_r\}$ be collections of $r$ indeterminates. Then let $R \coloneqq \Z[\mathbf{X}] \otimes \Sym(\mathbf{Y}) = \Z[x_1,\ldots,x_r] \otimes \Z[y_1,\ldots,y_r]^{\fk{S}_r}$ where $\deg(x_i) = \deg(y_i) = q^2$. We define a chain complex $C^r$ modeled on the $r$-dimensional cube $[0,1]^r$. For each vertex $v = (v_1,\ldots,v_r) \in \{0,1\}^r$, let $R(v) \coloneqq q^{2\sum_i iv_i}R$. Set \[
		C^r \coloneqq \bigoplus_{v \in \{0,1\}^{r}} t^{-\sum_i v_i}R(v).
	\]To define the differential $d$, first let $\cl{H}_r$ act $\Sym(\mathbf{Y})$-linearly on $R$ by its action on $\Z[\mathbf{X}]$. Set $Q_r \coloneqq e_r(\mathbf{Y} - x_r) \in R$. For $v_1,\ldots,v_{i-1},v_{i+1},\ldots,v_r\in \{0,1\}$, the component of $d$ from $R(v_1,\ldots,v_{i-1},1,v_{i+1},\ldots,v_r)$ to $R(v_1,\ldots,v_{i-1},0,v_{i+1},\ldots,v_r)$ is defined to be the composite\[
		\theta_{i}\,\theta_{i+1}\cdots\theta_{r-1}\,Q_r\,\hat{\theta}_{r-1}\cdots\hat{\theta}_{i+1}\,\hat{\theta}_i \qquad\text{where}\quad \theta_j \coloneqq \begin{cases}
			\partial_j & v_{j+1} = 0\\
			s_j & v_{j+1} = 1
		\end{cases}\quad\text{ and }\quad\hat{\theta}_j \coloneqq \begin{cases}
			s_j & v_{j+1} = 0\\
			\partial_j & v_{j+1} = 1.
		\end{cases}
	\]All other components are zero. See Figure~\ref{fig:examplesOfComplexC} for examples of $C^r$ when $r \in \{1,2,3\}$. The fact that $d^2 = 0$ follows from the center square case of the proof of \cite[Proposition 3.11]{wang2025minimalrickardcomplexesbraids}. 
\end{df}

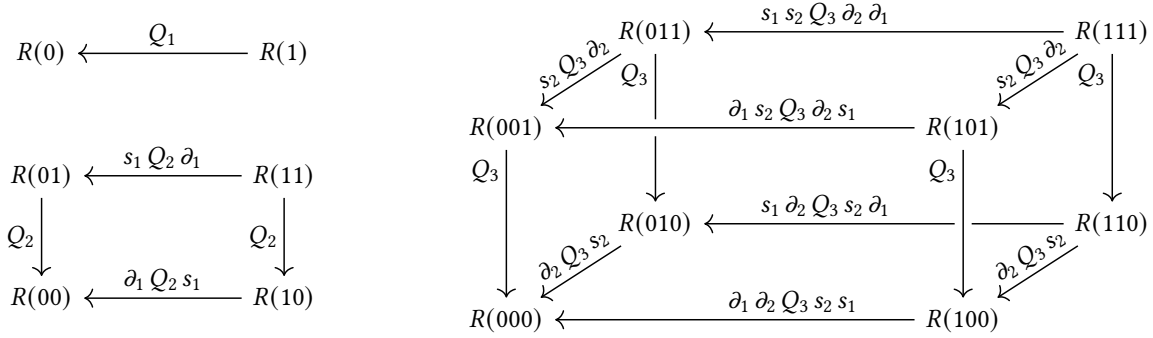
\begin{figure}[!ht]
	\centering
	\vspace{-15pt}\[
		\begin{tikzcd}[column sep=60pt,row sep=30pt]
			R(0) & R(1) \ar[l,swap,"\textstyle Q_1"]\\
			R(01) \ar[d,swap,"\textstyle Q_2"] & R(11) \ar[l,swap,"\textstyle{s_1\,Q_2\,\partial_1}"] \ar[d,swap,"\textstyle Q_2"] \\
			R(00) & R(10) \ar[l,swap,"\textstyle{\partial_1\,Q_2\,s_1}"]
		\end{tikzcd} \hspace{50pt} \begin{tikzcd}[column sep=20pt,row sep=20pt]
			& R(011) \ar[dd,swap,"\textstyle Q_3" yshift=20pt] \ar[dl,sloped,"\textstyle s_2\,Q_3\,\partial_2" xshift=1pt] & & & & & R(111) \ar[lllll,swap,"\textstyle s_1\,s_2\, Q_3\,\partial_2\,\partial_1" xshift=-22pt] \ar[dd,swap,"\textstyle Q_3" yshift=20pt] \ar[dl,sloped,"\textstyle s_2\,Q_3\,\partial_2" xshift=1pt]\\
			R(001) \ar[dd,swap,"\textstyle Q_3" yshift=20pt] & & & & & R(101) \ar[lllll,swap,crossing over,"\textstyle \partial_1\,s_2\, Q_3\,\partial_2\,s_1" xshift=22pt] &\\
			& R(010) \ar[dl,sloped,"\textstyle \partial_2\,Q_3\,s_2" xshift=1pt] & & & & & R(110) \ar[lllll,swap,"\textstyle s_1\,\partial_2\, Q_3\,s_2\,\partial_1" xshift=-22pt] \ar[dl,sloped,"\textstyle \partial_2\,Q_3\,s_2" xshift=1pt]\\
			R(000) & & & & & R(100) \ar[lllll,swap,"\textstyle \partial_1\,\partial_2\, Q_3\,s_2\,s_1" xshift=22pt] \ar[from=uu,crossing over,swap,"\textstyle Q_3" yshift=20pt]  &
		\end{tikzcd}\vspace{-10pt}
	\]
	\captionsetup{width=.8\linewidth}
	\caption{The complex $C^r$ for $r = 1$ in the top left, $r=2$ in the bottom left, and $r=3$ on the right. Cohomological degree shifts $t^{-\sum_i v_i}$ are omitted for brevity.}
	\label{fig:examplesOfComplexC}
\end{figure}

There is an action of the nil-Hecke algebra $\cl{H}_r$ on $C^r$ given by its action on each $R(v)$ for $v\in \{0,1\}^r$. However, this action is not by chain maps. We now define a different action of $\cl{H}_r$ on $C^r$ by chain maps. We let $x_1,\ldots,x_r,\partial_1,\ldots,\partial_{r-1},s_1,\ldots,s_{r-1}$ denote the action of $\cl{H}_r$ on each $R(v)$, and we let $X_1,\ldots,X_r,\Delta_1,\ldots,\Delta_{r-1},S_1,\ldots,S_{r-1}$ denote the action of $\cl{H}_r$ on $C^r$ by chain maps. 

\begin{df}\label{def:nilHeckeActionOnC}
 	Fix $r \ge 0$, and let $C^r$ be the complex of Definition~\ref{def:complexC}. \begin{itemize}[noitemsep]
 		\item[$\bullet$] For $i\in\{1,\ldots,r-1\}$, we define $\Delta_i\colon C^r \to C^r$ by components. Let $v$ be a vertex. If $v_i = v_{i+1}$, let the component of $\Delta_i$ from $R(v)$ to $R(v)$ be $\partial_i$. If $(v_i,v_{i+1}) = (0,1)$, let $w$ be the vertex obtained from $v$ by swapping its $i$th and $(i+1)$th entries. Let the component of $\Delta_i$ from $R(v)$ to $R(w)$ be $s_i$. All other components are zero.
 		\item[$\bullet$] For $i \in \{1,\ldots,r\}$, we define a map $\xi_i\colon C^r \to C^r$ and then set $X_i = x_i + \xi_i$. The map $x_i$ is just multiplication by $x_i$ on each $R(v)$. The map $\xi_i$ is defined by components as follows. Let $v$ be a vertex, and suppose $v_i = 1$ and $j$ is an index for which $v_j=0$ and $i < j$. Let $w$ be the vertex obtained from $v$ by swapping its $i$th and $j$th entries. The component of $\xi_i$ from $R(v)$ to $R(w)$ is defined to be \[
 			\theta_i\,\theta_{i+1}\cdots\theta_{j-2}\,s_{j-1}\,\hat{\theta}_{j-2}\cdots\hat{\theta}_{i+1}\,\hat{\theta}_i \qquad\text{where}\quad \theta_l \coloneqq \begin{cases}
				\partial_l & v_{l+1} = 0\\
				s_l & v_{l+1} = 1
			\end{cases}\quad\text{ and }\quad\hat{\theta}_l \coloneqq \begin{cases}
				s_l & v_{l+1} = 0\\
				\partial_l & v_{l+1} = 1.
			\end{cases}
 		\]Suppose $v_i = 0$ and $j$ is an index for which $v_j=1$ and $j < i$. Let $w$ be the vertex obtained from $v$ by swapping its $i$th and $j$th entries. The component of $\xi_i$ from $R(v)$ to $R(w)$ is defined to be \[
 			-\,\theta_j\,\theta_{j+1}\cdots\theta_{i-2}\,s_{i-1}\,\hat{\theta}_{i-2}\cdots\hat{\theta}_{j+1}\,\hat{\theta}_j \qquad\text{where}\quad \theta_l \coloneqq \begin{cases}
				\partial_l & v_{l+1} = 0\\
				s_l & v_{l+1} = 1
			\end{cases}\quad\text{ and }\quad\hat{\theta}_l \coloneqq \begin{cases}
				s_l & v_{l+1} = 0\\
				\partial_l & v_{l+1} = 1.
			\end{cases}
 		\]All other components are zero. 
 		\item[$\bullet$] For $i \in \{1,\ldots,r-1\}$, set $S_i = [\Delta_i, X_i] = \Delta_i\,X_i - X_i\,\Delta_i$. 
 	\end{itemize}See Figure~\ref{fig:nilHeckeAction} for explicit examples of $\Delta_i$ and $\xi_i$. The fact that $X_1,\ldots,X_r,\Delta_1,\ldots,\Delta_{r-1}$ define chain maps and satisfy the relations of the nil-Hecke algebra is proved in Appendix~\ref{sec:nil_hecke_action}. From this fact, it follows that $S_1,\ldots,S_{r-1}$ are chain maps that determine an action of the symmetric group $\fk{S}_r$ on $C^r$ by chain maps.
\end{df}

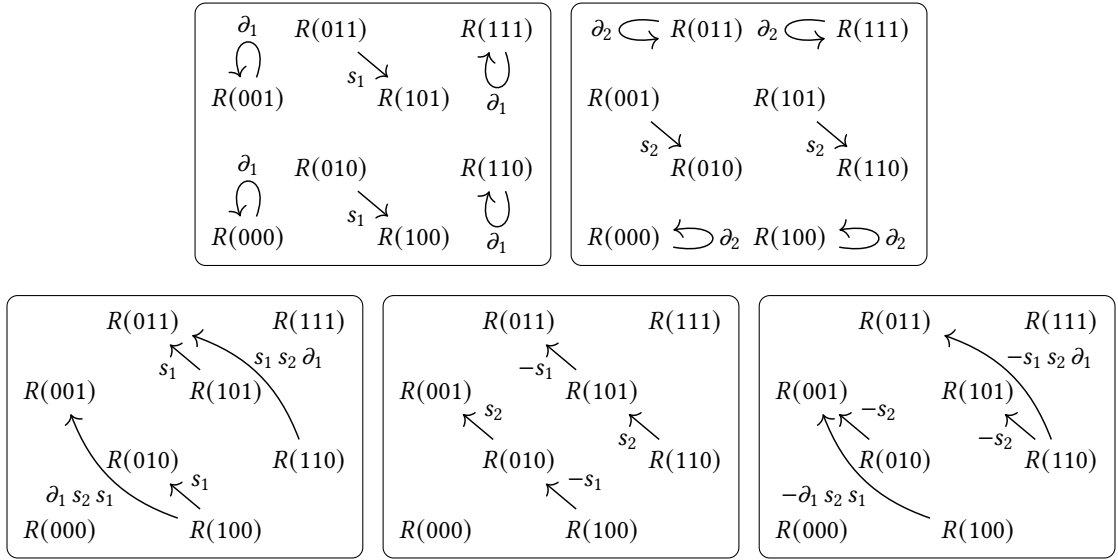
\begin{figure}[!ht]
	\centering
	\vspace{-10pt}\[
		\begin{tikzpicture}
			\node[draw,rounded corners,inner sep=2pt] at (0,0) {
				\begin{tikzcd}[column sep=-5pt,row sep=10pt]
					& R(011) \ar[dr,swap,"\textstyle s_1"] & & R(111) \ar[out=290,in=250,distance=50pt,loop,"\textstyle\partial_1"]\\
					R(001) \ar[out=70,in=110,distance=50pt,loop,swap,"\textstyle\partial_1"] & & R(101) &\\
					& R(010) \ar[dr,swap,"\textstyle s_1"] & & R(110) \ar[out=290,in=250,distance=50pt,loop,"\textstyle\partial_1"]\\
					R(000) \ar[out=70,in=110,distance=50pt,loop,swap,"\textstyle\partial_1"] & & R(100) &
				\end{tikzcd}
			};
			\node[draw,rounded corners,inner sep=2pt] at (5,0) {
				\begin{tikzcd}[column sep=-5pt,row sep=10pt]
					& R(011) \ar[out=170,in=190,distance=50pt,loop,swap,"\textstyle\partial_2"] & & R(111) \ar[out=170,in=190,distance=50pt,loop,swap,"\textstyle\partial_2"]\\
					R(001) \ar[rd,swap,"\textstyle s_2"] & & R(101) \ar[rd,swap,"\textstyle s_2"] &\\
					& R(010) & & R(110) \\
					R(000) \ar[out=-10,in=10,distance=50pt,loop,swap,"\textstyle\partial_2"] & & R(100) \ar[out=-10,in=10,distance=50pt,loop,swap,"\textstyle\partial_2"] &
				\end{tikzcd} 
			};
		\end{tikzpicture}
	\]\[
		\begin{tikzpicture}
			\node[draw,rounded corners,inner sep=2pt] at (0,0) {
				\begin{tikzcd}[column sep=-5pt,row sep=10pt]
					& R(011) & & R(111)\\
					R(001) & & R(101) \ar[lu,"\textstyle s_1"] &\\
					& R(010) & & R(110) \ar[lluu,bend right=25pt,swap,"\textstyle s_1\,s_2\,\partial_1" xshift=-4pt] \\
					R(000) & & R(100) \ar[lu,swap,"\textstyle s_1"] \ar[lluu,bend left=25pt,"\textstyle \partial_1\,s_2\,s_1" xshift=4pt] &
				\end{tikzcd}
			};
			\node[draw,rounded corners,inner sep=2pt] at (5,0) {
				\begin{tikzcd}[column sep=-5pt,row sep=10pt]
					& R(011) & & R(111)\\
					R(001) & & R(101) \ar[lu,"\textstyle -s_1"] &\\
					& R(010) \ar[lu,swap,"\textstyle s_2"] & & R(110) \ar[lu,"\textstyle s_2"] \\
					R(000) & & R(100) \ar[lu,swap,"\textstyle -s_1"] &
				\end{tikzcd} 
			};
			\node[draw,rounded corners,inner sep=2pt] at (10,0) {
				\begin{tikzcd}[column sep=-5pt,row sep=10pt]
					& R(011) & & R(111)\\
					R(001) & & R(101) &\\
					& R(010)  \ar[lu,swap,"\textstyle -s_2"] & & R(110) \ar[lu,"\textstyle -s_2"]  \ar[lluu,swap,bend right=25pt,"\textstyle -s_1\,s_2\,\partial_1" xshift=-4pt] \\
					R(000) & & R(100) \ar[lluu,bend left=25pt,"\textstyle -\partial_1\,s_2\,s_1" xshift=4pt] &
				\end{tikzcd}
			};
		\end{tikzpicture}\vspace{-10pt}
	\]
	\captionsetup{width=.8\linewidth}
	\caption{The chain endomorphisms $\Delta_1,\Delta_2$ of $C^3$ are in the top row, and the endomorphisms $\xi_1,\xi_2,\xi_3$ of $C^3$ are in the bottom row. The sum $X_i = x_i + \xi_i$ is a chain endomorphism.}
	\label{fig:nilHeckeAction}
\end{figure}

The polynomial ring $\Z[\mathbf{X}]$, as a graded module over $\Sym(\mathbf{X})$, is free with graded rank given by the $q$-factorial $(r)! = (r)(r-1)\cdots(1)$ where $(m)$ denotes the $q$-integer $1 + q^2 + \cdots + q^{2m-2} \in \Z[q]$. This $q$-polynomial is precisely the Poincar\'e polynomial of the flag manifold of $\C^r$. Given any composition $(r_1,\ldots,r_m)$ of $r$, the subgroup $\fk{S}_{r_1} \x \cdots \x \fk{S}_{r_m} \subseteq \fk{S}_r$ determines a ring of partially symmetric polynomials $\Sym(\mathbf{X}) \subseteq \Z[x_1,\ldots,x_r]^{\fk{S}_{r_1} \x \cdots \x\fk{S}_{r_m}}\subseteq \Z[\mathbf{X}]$ that is free over $\Sym(\mathbf{X})$ whose graded rank coincides with the Poincar\'e polynomial of the partial flag manifold $\mathrm{Fl}(r_1,\ldots,r_m;\C^r)$ of pairwise orthogonal $m$-tuples $(\Lambda_1,\ldots,\Lambda_m)$ of linear subspaces of $\C^r$ of dimensions $(r_1,\ldots,r_m)$. 

The purpose of constructing an action of the nil-Hecke algebra $\cl{H}_r$ on the complex $C^r$ is to show that it has the same structure. In particular, $C^r$ is isomorphic to the direct sum of $(r)!$ copies of a smaller complex, which we denote by $B^r$. Furthermore, it has a system of subcomplexes $B^{r_1,\ldots,r_m}$ indexed by compositions of $r$, each of which is a direct sum of shifted copies of $B^r$ according to the Poincar\'e polynomial of $\mathrm{Fl}(r_1,\ldots,r_m;\C^r)$. The subcomplex $B^{r_1,\ldots,r_m}$ is just the set of invariants under the action of $\fk{S}_{r_1} \x \cdots \x \fk{S}_{r_m} \subseteq \fk{S}_r$, but described explicitly in the following definition. 

\begin{df}\label{def:subcomplexB}
	Let $(r_1,\ldots,r_m)$ be a sequence of positive integers, with $r = r_1 + \cdots + r_m$, and let $C^r$ be the complex of Definition~\ref{def:complexC}. Let $B^{r_1,\ldots,r_m}$ be the following subcomplex of $C^r$, defined by specifying a submodule $B^{r_1,\ldots,r_m}(v) \subseteq R(v)$ for each vertex $v$. For each $m$-tuple $(a_1,\ldots,a_m)$ of integers satisfying $0 \leq a_i \leq r_i$, let $v \in \{0,1\}^r$ be the sequence $(1^{a_1},0^{r_1-a_1},1^{a_2},0^{r_2-a_2},\ldots,1^{a_m},0^{r_m-a_m})$. Let $B^{r_1,\ldots,r_m}(v) \subseteq R(v)$ be the subring of polynomials in $R(v)$ invariant under \[
		\fk{S}_{a_1} \x \fk{S}_{r_1-a_1} \x \cdots \x\fk{S}_{a_m} \x \fk{S}_{r_m-a_m} \subseteq \fk{S}_{r}
	\]where $\fk{S}_r$ acts on $R(v)$ in the standard way. If $v \in \{0,1\}^r$ is not of the form described, we set $B^{r_1,\ldots,r_m}(v) = 0$. The fact that $B^{r_1,\ldots,r_m}$ is a subcomplex follows from Lemma~\ref{lem:Bsubcomplexinvariants}.
\end{df}

\begin{lem}\label{lem:Bsubcomplexinvariants}
	For any composition $(r_1,\ldots,r_m)$ of $r$, the subcomplex of elements of $C^r$ that are invariant under the chain action of $\fk{S}_{r_1} \x \cdots \x \fk{S}_{r_m} \subseteq \fk{S}_r$ is precisely $B^{r_1,\ldots,r_m}$.
\end{lem}
\begin{proof}
	Let $J = \{1,\ldots,r-1\}\setminus\{r_1,r_1+r_2,\ldots,r_1+\cdots+r_{m-1}\}$ so that the simple transpositions $s_j$ for $j \in J$ are precisely the ones that generate $\fk{S}_{r_1} \x \cdots \x \fk{S}_{r_m}$. Note that $c \in C^r$ satisfies $S_jc = c$ if and only if $\Delta_j c = 0$ because $X_j - X_{j+1}$ is injective. It suffices to show that \[
		B^{r_1,\ldots,r_m} = \bigcap_{j\in J} \ker(\Delta_j).
	\]From the definition of $\Delta_j$, it is clear that $c = \bigoplus_v c_v\in C^r$ lies in the kernel of $\Delta_j$ if and only if each $c_v$ lies in the kernel of $\Delta_j$, so it suffices to prove the identity at each vertex $v$ of the cube. If $(v_j,v_{j+1}) = (0,1)$ for some $j\in J$, then $\Delta_j$ is injective on $R(v)$ so both $B^{r_1,\ldots,r_m}(v)$ and $R(v) \cap \bigcap_j \ker(\Delta_j)$ are zero. Assume that $(v_j,v_{j+1}) \neq (0,1)$ for every $j \in J$, and observe that $v$ must be of the form $(1^{a_1},0^{r_1-a_1},\ldots,1^{a_m},0^{r_m-a_m})$ for some sequence $(a_1,\ldots,a_m)$ of integers satisfying $0 \leq a_i \leq r_i$. For $j\in J$, the action of $\Delta_j$ on $R(v)$ is $\partial_j$ except when $j \in \{a_1,r_1+a_2,\ldots,r_1+\cdots+r_{m-1}+a_m\}$ in which case $R(v) \cap \ker(\Delta_j)= R(v)$. Hence, $R(v) \cap \bigcap_j \ker(\Delta_j)$ is the set of invariants under $s_j$ for $j \in J\setminus \{a_1,r_1+a_2,\ldots,r_1+\cdots+r_{m-1}+a_m\}$, which is precisely $B^{r_1,\ldots,r_m}(v)$.
\end{proof}

Given a partition $\lambda$ with at most $\min(k,N-k)$ parts, we now define a chain complex $\mathrm{F}_\lambda$ of $H^*(\BU(N))$-modules. 

\begin{df}\label{def:Flambda}
	Let $\lambda$ be a partition with at most $\min(k,N-k)$ parts, and let $(r_1,\ldots,r_m)$ be the multiplicities of its parts, with $r \coloneqq r_1 + \cdots + r_m$. Let $\mathbf{Z} = \{z_1,\ldots,z_{k-r}\}$ and $\mathbf{W} = \{w_1,\ldots,w_{N-k-r}\}$ be sets of indeterminates. Let $\mathrm{F}_\lambda$ be the complex \[
		\mathrm{F}_\lambda \coloneqq B^{r_1,\ldots,r_m} \otimes \Sym(\mathbf{Z}) \otimes \Sym(\mathbf{W})
	\]given simply by extending scalars, where $B^{r_1,\ldots,r_m}$ is defined in Definition~\ref{def:subcomplexB}. The action of $H^*(\BU(N))$ on $\mathrm{F}_\lambda$ is given by identifying $H^*(\BU(N)) = \Sym(\mathbf{X} \sqcup \mathbf{Y} \sqcup \mathbf{Z} \sqcup \mathbf{W}) \subseteq \Sym(\mathbf{X}) \otimes \Sym(\mathbf{Y}) \otimes \Sym(\mathbf{Z}) \otimes \Sym(\mathbf{W})$.
\end{df}
\begin{prop}\label{prop:Flambdafreeresolution}
	For any partition $\lambda$ with at most $\min(k,N-k)$ parts, the complex $\mathrm{F}_\lambda$ is a $H^*(\BU(N))$-free resolution of $H^*(\mathrm{B}J_\lambda)$, where $J_\lambda$ is given in Proposition~\ref{prop:BlambdaAsUnitaryOrbit}.
\end{prop}
\begin{proof}
	The fact that $\mathrm{F}_\lambda$ is free over $H^*(\BU(N))$ follows from the fact that $B^{r_1,\ldots,r_m}$ is free over $\Sym(\mathbf{X}) \otimes \Sym(\mathbf{Y})$ and the fact that $\Sym(\mathbf{X}) \otimes\Sym(\mathbf{Y}) \otimes \Sym(\mathbf{Z}) \otimes \Sym(\mathbf{W})$ is free over $\Sym(\mathbf{X} \sqcup \mathbf{Y} \sqcup\mathbf{Z} \sqcup\mathbf{W}) = H^*(\BU(N))$. 

	Write $\mathbf{X} = \mathbf{X}_1\sqcup \ldots \sqcup \mathbf{X}_m$ and $\mathbf{Y} = \mathbf{Y}_1\sqcup \ldots \sqcup \mathbf{Y}_m$ so that $\Z[\mathbf{X}]^{\fk{S}_{r_1} \x \cdots \x \fk{S}_{r_m}} = \Sym(\mathbf{X}_1) \otimes \cdots \otimes \Sym(\mathbf{X}_m)$ and $\Z[\mathbf{Y}]^{\fk{S}_{r_1} \x \cdots \x \fk{S}_{r_m}} = \Sym(\mathbf{Y}_1) \otimes \cdots \otimes \Sym(\mathbf{Y}_m)$. In particular, $|\mathbf{X}_j| = |\mathbf{Y}_j| = r_j$ for $j\in\{1,\ldots,m\}$. From the explicit description of $J_\lambda$ given in Proposition~\ref{prop:BlambdaAsUnitaryOrbit}, note that \[
		H^*(\mathrm{B}J_\lambda) = \frac{\Sym(\mathbf{X}_1) \otimes \cdots \otimes \Sym(\mathbf{X}_m) \otimes \Sym(\mathbf{Y}_1) \otimes \cdots \otimes \Sym(\mathbf{Y}_m) \otimes \Sym(\mathbf{Z}) \otimes \Sym(\mathbf{W})}{(e_i(\mathbf{X}_j) - e_i(\mathbf{Y}_j) \:|\: j \in \{1,\ldots,m\} \text{ and }i\in\{1,\ldots,r_j\})}
	\]as a module over $H^*(\BU(N)) = \Sym(\mathbf{X} \sqcup \mathbf{Y} \sqcup\mathbf{Z} \sqcup\mathbf{W})$. It therefore suffices to show that $B^{r_1,\ldots,r_m}$ is a resolution of \[
		\frac{\Sym(\mathbf{X}_1) \otimes \cdots\otimes \Sym(\mathbf{X}_m) \otimes \Sym(\mathbf{Y}_1) \otimes \cdots\otimes \Sym(\mathbf{Y}_m)}{(e_i(\mathbf{X}_j) - e_i(\mathbf{Y}_j) \:|\: j \in \{1,\ldots,m\} \text{ and }i\in\{1,\ldots,r_j\})} = \frac{\Z[\mathbf{X}]^{\fk{S}_{r_1} \x \cdots \x \fk{S}_{r_m}} \otimes \Sym(\mathbf{Y})}{(e_i(\mathbf{X}) - e_i(\mathbf{Y})\:|\: i\in\{1,\ldots,r\})}
	\]as a $\Sym(\mathbf{X}) \otimes \Sym(\mathbf{Y})$-module. This module is just $\Z[\mathbf{X}]^{\fk{S}_{r_1} \x \cdots \x \fk{S}_{r_m}}$ where $e_i(\mathbf{Y})$ acts by $e_i(\mathbf{X})$. 

	We first prove the result in the case that $(r_1,\ldots,r_m)=(1^r)$ so $B^{r_1,\ldots,r_m} = C^r$. Define $Q_{r,j} \in R = \Z[\mathbf{X}] \otimes \Sym(\mathbf{Y})$ for $j \in \{0,\ldots,r-1\}$ by setting $Q_{r,0} \coloneqq Q_r = e_r(\mathbf{Y} - x_r)$ and $Q_{r,j} = \partial_{r-j}(Q_{r,j-1})$ for $j > 0$. It is straightforward to show inductively that \[
		Q_{r,j} = e_{r-j}(\mathbf{Y} - \{x_{r-j},\ldots,x_r\}) = e_{r-j}(\mathbf{Y}) - e_{r-j-1}(\mathbf{Y})h_1(x_{r-j},\ldots,x_r) + \cdots + (-1)^{r-j}h_{r-j}(x_{r-j},\ldots,x_r).
	\]By considering the leading term of this expression, it follows that $Q_{r,0},Q_{r,1},\ldots,Q_{r,r-1} \in R$ form a regular sequence. Furthermore, the ideal generated by these $r$ polynomials is precisely $(e_i(\mathbf{X}) - e_i(\mathbf{Y}) \:|\: i\in\{1,\ldots,r\})$. Indeed, \[
		Q_{r,j} = e_{r-j}(\mathbf{Y} - \mathbf{X} + \{x_1,\ldots,x_{r-j-1}\}) = e_{r-j}(\mathbf{Y} - \mathbf{X}) + \sum_{i=1}^{r-j-1} e_{r-j-i}(\mathbf{Y} - \mathbf{X})e_i(x_1,\ldots,x_{r-j-1})
	\]shows that $(Q_{r,0},\ldots,Q_{r,r-1}) = (e_i(\mathbf{Y} -\mathbf{X}) \:|\: i\in\{1,\ldots,r\}) = (e_i(\mathbf{Y}) - e_i(\mathbf{X}) \:|\: i\in\{1,\ldots,r\})$. It therefore suffices to prove that the homology of $C^r$ is $R/(Q_{r,0},\ldots,Q_{r,r-1})$, supported in degree $0$. To do this, we iteratively apply the spectral sequence associated to a bicomplex. First, we take homology with respect to the components of the differential of $C^r$ that decrease the $r$th entry of $v$. Each such component is just multiplication by $Q_r = Q_{r,0}$ which is injective, so the homology is the direct sum $\bigoplus_{v \in \{0,1\}^{r-1} \x \{0\}} t^{-\sum_i v_i} \, R(v)/(Q_{r,0})$ over the vertices of the face $[0,1]^{r-1} \x \{0\}$. Next, we take homology with respect to the components of the induced differential that decrease the $(r-1)$th entry of $v$. Each such component is the map $R/(Q_{r,0}) \to R/(Q_{r,0})$ given by $\partial_{r-1}\,Q_{r,0}\,s_{r-1} = Q_{r,1} - Q_{r,0}\,\partial_{r-1} \equiv Q_{r,1} \:\bmod (Q_{r,0})$. This map is again injective because $Q_{r,0},Q_{r,1},\ldots,Q_{r,r-1} \in R$ is a regular sequence, so its homology is $\bigoplus_{v\in\{0,1\}^{r-2} \x \{0\}^2} t^{-\sum_i v_i} \, R(v)/(Q_{r,0},Q_{r,1})$. Continuing in this way using the identity $\partial_{r-j}\cdots\partial_{r-1}\,Q_{r,0}\,s_{r-1}\cdots s_{r-j} \equiv Q_{r,j} \:\bmod (Q_{r,0},\ldots,Q_{r,j-1})$, which is straightforward to show, we conclude that the homology of $C^r$ is $R/(Q_{r,0},\ldots,Q_{r,r-1})$, supported in degree $0$.

	Let $(r_1,\ldots,r_m)$ be any composition of $r$. We claim that the subcomplex $B^{r_1,\ldots,r_m} \subseteq C^r$ is in fact a direct summand of $C^r$. Associated to the subgroup $\fk{S}_{r_1} \x \cdots \x \fk{S}_{r_m} \subseteq \fk{S}_r$ is a standard idempotent in $\cl{H}_r$ whose image in $C^r$ is precisely $B^{r_1,\ldots,r_m}$. To be explicit, for $j \in\{1,\ldots,m\}$ and $i\in\{1,\ldots,r_j\}$, let $X_{j,i}$ be shorthand notation for $X_{r_1 + \cdots + r_{j-1} + i}$. Set $E = E_1\,E_2\,\cdots\,E_m$ where \begin{align*}
		E_j &= (\Delta_{j,1})\,(\Delta_{j,2}\,\Delta_{j,1})\cdots(\Delta_{j,r_j-1}\cdots\Delta_{j,1})\,(X_{j,1})\,(X_{j,2}\,X_{j,1})\cdots(X_{j,r_j-1}\cdots X_{j,1})\\
		&= (\Delta_{j,1}\,X_{j,1})\,(\Delta_{j,2}\,X_{j,2}\,\Delta_{j,1}\,X_{j,1})\cdots(\Delta_{j,r_j-1}\,X_{j,r_j-1}\cdots\Delta_{j,1}\,X_{j,1}).
	\end{align*}We briefly explain a proof of the standard fact that $E$ is idempotent with image $B^{r_1,\ldots,r_m}$. First, the fact that these two expressions agree simply follows from repeated applications of the identity that $\Delta_{j,i}$ commutes with any monomial in which the powers of $X_{j,i}$ and $X_{j,i+1}$ are the same. It is straightforward to check from the first expression for $E_j$ that $\Delta_{j,i}\,E_j = 0$ for $i \in \{1,\ldots,r_j-1\}$, from which it follows that the image of $E$ is contained in $B^{r_1,\ldots,r_m}$. From the second expression for $E_j$, we may replace each of the $r_j(r_j-1)/2$ pairs $\Delta_{j,i}\,X_{j,i}$ as a two-term sum $\Id + \,X_{j,i+1}\,\Delta_{j,i}$ so that the expanded product is a sum over $\smash{2^{r_j(r_j-1)/2}}$ terms. One of these terms is the identity, and every other term is a composite where the rightmost non-identity map is $\Delta_{j,i}$ for some $i$. Hence, when applied to an element in $\ker(\Delta_{j,1})\cap\cdots\cap\ker(\Delta_{j,r_j-1})$, all terms except for the identity are zero. So $E$ restricts to the identity map on $B^{r_1,\ldots,r_m}$. 

	The chain map $E\colon C^r \to B^{r_1,\ldots,r_m}$ establishes that $B^{r_1,\ldots,r_m}$ is a direct summand of $C^r$, so $B^{r_1,\ldots,r_m}$ is also a free resolution. Its degree $0$ homology is the image of $B^{r_1,\ldots,r_m}(0^r) \to R(0^r) \to R(0^r)/(e_i(\mathbf{X}) - e_i(\mathbf{Y}) \:|\: i\in\{1,\ldots,r\})$, which, by definition of $B^{r_1,\ldots,r_m}(0^r)$, is $(\Z[\mathbf{X}]^{\fk{S}_{r_1} \x \cdots \x \fk{S}_{r_m}} \otimes \Sym(\mathbf{Y}))/(e_i(\mathbf{X}) - e_i(\mathbf{Y}) \:|\: i\in\{1,\ldots,r\})$ as required.
\end{proof}

We now turn to link homology. Following \cite{MR4903257,wang2025minimalrickardcomplexesbraids}, the Rickard complexes associated to crossings between strands labeled $k$ are \begin{align*}
	\quad\begin{gathered}
		\labellist
		\pinlabel $k$ at -2 5.5
		\pinlabel $k$ at -2 0.5
		\pinlabel $k$ at 14 0.5
		\pinlabel $k$ at 14 5.5
		\endlabellist
		\includegraphics[width=.06\textwidth]{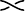}
		\vspace{-3pt}
	\end{gathered} \quad &\coloneqq \begin{tikzcd}[ampersand replacement=\&,column sep=30pt]
		W_0 \& t^{-1}q^1\, W_1 \ar[l,swap,"\zeta_{01}"] \& \cdots \ar[l,swap,"\zeta_{12}"] \& t^{-k}q^k\,W_k \ar[l,swap,"\zeta_{(k-1)k}" xshift=2pt]
	\end{tikzcd}\\
	\quad\begin{gathered}
		\labellist
		\pinlabel $k$ at -2 5.5
		\pinlabel $k$ at -2 0.5
		\pinlabel $k$ at 14 0.5
		\pinlabel $k$ at 14 5.5
		\endlabellist
		\includegraphics[width=.06\textwidth]{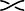}
		\vspace{-3pt}
	\end{gathered} \quad &\coloneqq \begin{tikzcd}[ampersand replacement=\&,column sep=30pt]
		t^kq^{-k} \,W_k \& \cdots \ar[l,swap,"\zeta_{k(k-1)}"] \& t^{1}q^{-1}\,W_1 \ar[l,swap,"\zeta_{21}"] \& W_0 \ar[l,swap,"\zeta_{10}" xshift=2pt]
	\end{tikzcd}
\end{align*}\[
	W_r \coloneqq \quad\begin{gathered}
	\labellist
	\small
	\pinlabel $k$ at -.8 4.8
	\pinlabel $k$ at -.8 0.4
	\pinlabel $r$ at 2.1 2.4
	\pinlabel $k-r$ at 5.7 6
	\pinlabel $k$ at 12.2 4.8
	\pinlabel $k$ at 12.2 0.4
	\pinlabel $r$ at 9.2 2.4
	\pinlabel $k+r$ at 5.7 -1
	\endlabellist
	\includegraphics[width=.12\textwidth]{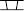}
	\vspace{-3pt}
\end{gathered}\quad \hspace{30pt} \zeta_{r(r+1)} = \hspace{20pt}\begin{gathered}
		\labellist
		\small
		\pinlabel $W_{r+1}$ at -2 11
		\pinlabel $W_r$ at -2 2
		\endlabellist
		\includegraphics[width=.163\textwidth]{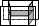}
		\vspace{-3pt}
	\end{gathered} \hspace{30pt} \zeta_{(r+1)r} = \hspace{20pt}\begin{gathered}
		\labellist
		\small
		\pinlabel $W_{r}$ at -1.5 11
		\pinlabel $W_{r+1}$ at -2.5 2
		\endlabellist
		\includegraphics[width=.163\textwidth]{zetaFoam}
		\vspace{-3pt}
	\end{gathered}
\]where braids and webs flow from right to left, and foams are read from top to bottom. The main construction of \cite{wang2025minimalrickardcomplexesbraids} is an explicit minimal complex $\sr{F}^m(\sr{P}_k)$ that is homotopy equivalent to the $m$th power of the positive crossing \[
	\sr{F}^m(\sr{P}_k) \simeq \quad\begin{gathered}
		\labellist
		\pinlabel $k$ at -2 5.5
		\pinlabel $k$ at -2 0.5
		\endlabellist
		\includegraphics[width=.06\textwidth]{1halfTwist}
		\vspace{-3pt}
	\end{gathered}\hspace{3pt}\begin{gathered}
		\cdots
	\end{gathered}\hspace{3pt}\begin{gathered}
		\labellist
		\pinlabel $k$ at 14 0.5
		\pinlabel $k$ at 14 5.5
		\endlabellist
		\includegraphics[width=.06\textwidth]{1halfTwist}
		\vspace{-3pt}
	\end{gathered}\quad
\]As explained in \cite{wang2025minimalrickardcomplexesbraids}, we may view $\sr{F}^m(\sr{P}_k)$ as a complex of singular Soergel bimodules or, as we shall do here, as a complex of $\sl(N)$ webs and foams. Up to a grading shift $q^{-k(N-k)(2m+1)}$ arising from the blackboard framing, the $k$-colored $\sl(N)$ complex of the torus knot $T(2,2m+1)$, denoted by $\mathrm{KRC}_{k,N}(T(2,2m+1))$, is homotopy equivalent to the braid closure of $\sr{F}^{2m+1}(\sr{P}_k)$. As in \cite{MR4903257}, it is preferable to work instead with the braid closure of $\sr{F}^{2m}(\sr{P}_k)$ composed with an additional positive half twist, which we refer to as the \textit{twist closure} of $\sr{F}^{2m}(\sr{P}_k)$. \[
	q^{-k(N-k)(2m+1)}\,\mathrm{KRC}_{k,N}(T(2,2m+1)) \:\simeq \qquad\begin{gathered}
		\labellist
		\pinlabel $\sr{F}^{2m+1}(\sr{P}_k)$ at 10.7 6.8
		\pinlabel $k$ at -1.1 10.3
		\pinlabel $k$ at -1.1 2.7
		\endlabellist
		\includegraphics[width=.2\textwidth]{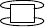}
	\end{gathered} \quad \simeq \qquad\begin{gathered}
		\labellist
		\pinlabel $\sr{F}^{2m}(\sr{P}_k)$ at 11.3 6.8
		\pinlabel $k$ at -1.1 10.3
		\pinlabel $k$ at -1.1 2.7
		\endlabellist
		\includegraphics[width=.218\textwidth]{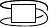}
	\end{gathered}\quad
\]As indicated by the notation, $\sr{F}^m(\sr{P}_k)$ for $m \ge 0$ is an exhaustive filtration of a complex $\sr{P}_k$, and the $k$-colored $\sl(N)$ complex of $T(2,\infty)$ is homotopy equivalent to the braid closure of $q^{2k(N-k)} \,\sr{P}_k$.

We recall some aspects of the construction of $\sr{F}^{m}(\sr{P}_k)$ and refer to \cite[section 4]{wang2025minimalrickardcomplexesbraids} for details. First, there is an auxiliary complex $\sr{K}_k$ modeled on the cube $[0,3]^k$. For each $\varepsilon \in [0,3]^k \cap \Z^k$, we have an object $V(\varepsilon) = q^{G(\varepsilon)}\,V_{r(\varepsilon)}$ where $r(\varepsilon)$ is the number of entries of $\varepsilon$ that are equal to $1$ or $2$ and $V_r$ is the web \vspace{0pt}\[
	V_r = \qquad \begin{gathered}
		\labellist
		\pinlabel $k$ at -0.5 0.2
		\pinlabel $k$ at -0.5 7.7
		\pinlabel $k$ at 22.2 7.7
		\pinlabel $k$ at 22.2 0.2
		\pinlabel $1$ at 1.3 6
		\pinlabel $1$ at 2.85 6
		\pinlabel $\cdots$ at 4.6 5.8
		\pinlabel $1$ at 5.8 6
		\pinlabel $1$ at 13.1 1
		\pinlabel $1$ at 13.1 2.5
		\pinlabel $\vdots$ at 13.1 4.5
		\pinlabel $1$ at 13.1 5.5
		\pinlabel $1$ at 13.1 7
		\pinlabel $r$ at 20.5 6
		\pinlabel $k+r$ at 13.1 {-.6}
		\endlabellist
		\includegraphics[width=.45\textwidth]{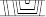}
	\end{gathered}\vspace{3pt}
\]The grading-shift function $G$ is defined in \cite[section 3.3]{wang2025minimalrickardcomplexesbraids}. Then $\sr{K}_k = \bigoplus_\varepsilon t^{-|\varepsilon|}V(\varepsilon)$. The nonzero components of the differential are along edges of $[0,3]^k$, which decrease an entry of $\varepsilon$ by $1$, and explicit formulas are given in \cite[section 3.4]{wang2025minimalrickardcomplexesbraids}. We only recall the formula for components of the following special form. Fix a number $r \in \{1,\ldots,k\}$, an index $i \in \{1,\ldots,r\}$, and $\varepsilon_1,\ldots,\varepsilon_{i-1},\varepsilon_{i+1},\ldots,\varepsilon_r \in\{1,2\}$. Let $\varepsilon^j = (\varepsilon_1,\ldots,\varepsilon_{i-1},j,\varepsilon_{i+1},\ldots,\varepsilon_r,0,\ldots,0) \in [0,3]^k \cap \Z^k$ for $j \in\{1,2\}$. Note that $r = r(\varepsilon^1) = r(\varepsilon^2)$, so $V(\varepsilon^2)$ and $V(\varepsilon^1)$ are $q$-shifts of the same web $V_r$, shown above. The component of the differential from $V(\varepsilon^2)$ to $V(\varepsilon^1)$ is a composite of basic endomorphisms that we now describe. There is an action of $\Z[x_1,\ldots,x_k]$ on $V_r$ where $x_i$ acts by the dot map on the $i$th edge labeled $1$ from left to right (the first $r$ connect the top horizontal line to the bottom, the next $k-r-1$ have both endpoints on the top line, and the last edge is completely horizontal). This action extends to an action of the subalgebra of the nil-Hecke algebra $\cl{H}_k$ generated by $x_1,\ldots,x_k$ and $\partial_1,\ldots,\partial_{r-1},\partial_{r+1},\ldots,\partial_{k-1}$. Letting $\mathbf{Y}$ denote the set of $r$ indeterminates associated to the rightmost edge labeled $r$, we have an endomorphism $Q^*_r \coloneqq e_r(\mathbf{Y} - x_r)$ of $V_r$. The component of the differential from $V(\varepsilon^2)$ to $V(\varepsilon^1)$ is \[
	\theta_{i}\,\cdots\,\theta_{r-1}\,Q^*_r\,\hat{\theta}_{r-1}\,\cdots\,\hat{\theta}_i \qquad\qquad \theta_j \coloneqq \begin{cases}
		\partial_j & \varepsilon_{j+1} = 1\\
		s_j & \varepsilon_{j+1} = 2
	\end{cases} \qquad \hat{\theta}_j \coloneqq \begin{cases}
		s_j & \varepsilon_{j+1} = 1\\
		\partial_j & \varepsilon_{j+1} = 2
	\end{cases}
\]times the sign $(-1)^{(\varepsilon_{i+1}+1)+\cdots+(\varepsilon_r+1)}$. The sign makes every square of $\sr{K}_k$ commute. We make every square anticommute by adding signs in the following way. If $\varepsilon'$ is obtained by decreasing the $i$th entry of $\varepsilon$ by $1$, then multiply the component of the differential from $V(\varepsilon)$ to $V(\varepsilon')$ by $(-1)^{(\varepsilon_{i+1}+1)+\cdots+(\varepsilon_k+1)}$. The sign-corrected component of the differential from $V(\varepsilon^2)$ to $V(\varepsilon^1)$ is the formula displayed above times $(-1)^{k-r}$. 

The complex $\sr{F}^{m}(\sr{P}_k)$ is modeled on the set of partitions $\mu$ with at most $k$ parts and $|\mu|_\infty \leq m$. For each such partition $\mu$, we have an object $W(\mu) = q^{H(\mu)}\,\smash{W_{r(\mu)}^{g(\mu)}}$ where $g(\mu) = (r_1,\ldots,r_m)$ is the tuple of multiplicities of the parts of $\mu$, and $r(\mu) = r_1 + \cdots + r_m$ is the number of parts of $\mu$. The web $W^{r_1,\ldots,r_m}_r$ is \vspace{0pt}\[
	W_r^{r_1,\ldots,r_m} \coloneqq \qquad \begin{gathered}
		\labellist
		\pinlabel $k$ at -0.5 0.3
		\pinlabel $k$ at -0.5 7.6
		\pinlabel $k$ at 14.8 7.6
		\pinlabel $k$ at 14.8 0.3
		\pinlabel $r_1$ at 1.2 6
		\pinlabel $r_2$ at 2.75 6
		\pinlabel $\cdots$ at 4.35 6
		\pinlabel $r_m$ at 5.6 6
		\pinlabel $k-r$ at 9.4 6.9
		\pinlabel $r$ at 13 6
		\pinlabel $k+r$ at 9.4 {0.9}
		\endlabellist
		\includegraphics[width=.295\textwidth]{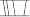}
	\end{gathered}\vspace{0pt}
\]The grading-shift function $H$ is $H(\mu) = \left(\sum_{i=1}^r 2i\mu_i \right) - \frac12(r(\mu)^2 + \sum_{j=1}^m r_j^2)$. The complex is the direct sum $\sr{F}^{2m}(\sr{P}_k) = \smash{\bigoplus_\mu t^{-|\mu|}\,W(\mu)}$. The differential is defined using $\sr{K}_k$. We associate a vertex $\varepsilon(\mu) \in [0,3]^k \cap \Z^k$ of the cube to each partition $\mu$ with at most $k$ parts. Define $\Z_{\ge 0} \to \{0,1,2\}$ by sending $0$ to $0$, odd numbers to $1$, and positive even numbers to $2$. Let $\varepsilon(\mu)$ be obtained from $\mu$ by applying this function to each entry, and note that $r(\mu) = r(\varepsilon(\mu))$. There is a natural inclusion map $W^{r_1,\ldots,r_m}_r \to V_r$ that is modeled on the inclusion of $\Z[x_1,\ldots,x_k]^{\fk{S}_{r_1} \x \cdots \x \fk{S}_{r_m} \x \fk{S}_{k-r}}$ into $\Z[x_1,\ldots,x_k]$.

Let $i$ be an index, let $\mu'$ be obtained from $\mu$ by decreasing its $i$th entry $\mu_i$ by $1$, and assume that $\mu'$ is a partition. So either $i = k$ and $\mu_k > 0$ or $i < k$ and $\mu_i > \mu_{i+1}$. The component of the differential of $\sr{F}^{m}(\sr{P}_k)$ from $W(\mu)$ to $W(\mu')$ is the unique map that makes \vspace{-5pt} \[
	\begin{tikzcd}[row sep = 15pt]
		V(\varepsilon(\mu')) & V(\varepsilon(\mu)) \ar[l]\\
		W(\mu') \ar[u] & W(\mu) \ar[u] \ar[l]
	\end{tikzcd}
\]commute. The vertical arrows are the natural inclusion maps. If $\varepsilon(\mu')$ is obtained from $\varepsilon(\mu)$ by decreasing its $i$th entry by $1$, the top horizontal arrow is just the differential in $\sr{K}_k$. This occurs when $\mu_i$ is even or when $\mu_i$ is $1$. The only other possibility is that $\mu_i$ is odd and greater than $1$, in which case the top horizontal map is from the vertex $\varepsilon(\mu)$ whose $i$th entry is $1$ to the vertex $\varepsilon(\mu')$ whose $i$th entry is $2$. It is defined to be the following composite. We first apply the component of the differential of $\sr{K}_k$ that lowers the $i$th entry of $\varepsilon(\mu)$ from $1$ to $0$. We then identify the web associated to this vertex with the one where $0$ is replaced by $3$. We then apply the component of the differential of $\sr{K}_k$ that lowers the $i$th entry from $3$ to $2$. See \cite[section 4.4]{wang2025minimalrickardcomplexesbraids} for details.

\begin{prop}\label{prop:slNcomplexT2m}
	The $k$-colored $\sl(N)$ complexes of $T(2,\infty)$ and $T(2,2m+1)$ are homotopy equivalent to \begin{align*}
		\mathrm{KRC}_{k,N}(T(2,\infty)) \:&\simeq \: \bigoplus_{\lambda}  t^{-2|\lambda|}q^{4\sum_i i\lambda_i} \,\mathrm{F}_\lambda \otimes_{H^*(\BU(N))} \Z\\[2pt]
		\mathrm{KRC}_{k,N}(T(2,2m+1)) \:&\simeq \: q^{k(N-k)(2m-1)}\,\bigoplus_{|\lambda|_\infty \leq m}  t^{-2|\lambda|}q^{4\sum_i i\lambda_i} \,\mathrm{F}_\lambda \otimes_{H^*(\BU(N))} \Z
	\end{align*}
	where the direct sums are over partitions $\lambda$ with at most $\min(k,N-k)$ parts. 
\end{prop}

We summarize the proof of Proposition~\ref{prop:slNcomplexT2m}. The $k$-colored $\sl(N)$ complex of $T(2,2m+1)$ is homotopy equivalent to the twist closure of $\sr{F}^{2m}(\sr{P}_k)$, which we denote by $\ol{\sr{F}^{2m}(\sr{P}_k)T} = \smash{\bigoplus_{|\mu|_\infty \leq 2m} t^{-|\mu|}\,\ol{W(\mu)T}}$. The direct sum is over the set of partitions $\mu$ with at most $k$ parts and $|\mu|_\infty \leq 2m$. When $k = m = 2$, we have \vspace{-5pt}\[
	\begin{tikzcd}[column sep=25pt,row sep=15pt]
		& & & & t^{-8}\,\ol{W(44)T} \ar[d]\\
		& & & t^{-6}\,\ol{W(33)T} \ar[d] & t^{-7}\,\ol{W(43)T} \ar[l] \ar[d]\\
		& & t^{-4}\,\ol{W(22)T} \ar[d] & t^{-5}\,\ol{W(32)T} \ar[l] \ar[d] & t^{-6}\,\ol{W(42)T} \ar[l] \ar[d]\\
		& t^{-2}\,\ol{W(11)T} \ar[d] & t^{-3}\,\ol{W(21)T} \ar[l] \ar[d] & t^{-4}\,\ol{W(31)T} \ar[l] \ar[d] & t^{-5}\,\ol{W(41)T} \ar[l] \ar[d]\\
		\ol{W(00)T} & t^{-1}\,\ol{W(10)T} \ar[l] & t^{-2}\,\ol{W(20)T} \ar[l] & t^{-3}\,\ol{W(30)T} \ar[l] & t^{-4}\,\ol{W(40)T} \ar[l]
	\end{tikzcd}
\]
We show that the complex $\ol{W(\mu)T}$ deformation retracts onto a complex $U(\mu)$ supported in a single grading. These deformation retracts over all $\mu$ provide a deformation retract of the entire complex $\ol{\sr{F}^{2m}(\sr{P}_k)T}$ onto a complex with objects $\bigoplus_{|\mu|_\infty\leq 2m}t^{-|\mu|}U(\mu)$. A priori, there can be nonzero components of the differential $U(\mu) \to U(\mu')$ whenever $\mu' < \mu$ entrywise. We show that the component is nonzero only if $\mu'$ is obtained by decreasing an \textit{even} entry of $\mu$ by $1$. \[
	\begin{tikzpicture}[every node/.style={inner sep=2pt}, x=80pt, y=30pt]
		\node (a00) at (0,0) {$U(00)$};
		\node (a10) at (1,0) {$t^{-1}\,U(10)$};
		\node (a20) at (2,0) {$t^{-2}\,U(20)$};
		\node (a30) at (3,0) {$t^{-3}\,U(30)$};
		\node (a40) at (4,0) {$t^{-4}\,U(40)$};
		\node (a11) at (1,1) {$t^{-2}\,U(11)$};
		\node (a21) at (2,1) {$t^{-3}\,U(21)$};
		\node (a31) at (3,1) {$t^{-4}\,U(31)$};
		\node (a41) at (4,1) {$t^{-5}\,U(41)$};
		\node (a22) at (2,2) {$t^{-4}\,U(22)$};
		\node (a32) at (3,2) {$t^{-5}\,U(32)$};
		\node (a42) at (4,2) {$t^{-6}\,U(42)$};
		\node (a33) at (3,3) {$t^{-6}\,U(33)$};
		\node (a43) at (4,3) {$t^{-7}\,U(43)$};
		\node (a44) at (4,4) {$t^{-8}\,U(44)$};

		\draw[{<[scale=1.5]}-] (a33)--(a43);
		\draw[{<[scale=1.5]}-] (a32)--(a42);
		\draw[{<[scale=1.5]}-] (a11)--(a21);
		\draw[{<[scale=1.5]}-] (a31)--(a41);
		\draw[{<[scale=1.5]}-] (a10)--(a20);
		\draw[{<[scale=1.5]}-] (a30)--(a40);

		\draw[{<[scale=1.5]}-] (a43)--(a44);
		\draw[{<[scale=1.5]}-] (a21)--(a22);
		\draw[{<[scale=1.5]}-] (a31)--(a32);
		\draw[{<[scale=1.5]}-] (a41)--(a42);

		\path[
		      draw,
		      rounded corners, 
		      line width=0.6pt
		    ] 
		      ($ (a33.north west) + (-2pt, 0pt) $) --
		      ($ (a33.south west) + (-2pt,-2pt) $) --
		      ($ (a43.south east) + ( 2pt,-2pt) $) --
		      ($ (a44.north east) + ( 2pt, 2pt) $) -- 
		      ($ (a44.north west) + ( 0pt, 2pt) $) -- cycle;
		\path[
		      draw,
		      rounded corners, 
		      line width=0.6pt
		    ] 
		      ($ (a11.north west) + (-2pt, 0pt) $) --
		      ($ (a11.south west) + (-2pt,-2pt) $) --
		      ($ (a21.south east) + ( 2pt,-2pt) $) --
		      ($ (a22.north east) + ( 2pt, 2pt) $) -- 
		      ($ (a22.north west) + ( 0pt, 2pt) $) -- cycle;
		\path[
		      draw,
		      rounded corners, 
		      line width=0.6pt
		    ] 
		      ($ (a31.south west) + (-2pt,-2pt) $) --
		      ($ (a41.south east) + ( 2pt,-2pt) $) --
		      ($ (a42.north east) + ( 2pt, 2pt) $) -- 
		      ($ (a32.north west) + (-2pt, 2pt) $) -- cycle;
		\path[
		      draw,
		      rounded corners, 
		      line width=0.6pt
		    ] 
		      ($ (a30.south west) + (-2pt,-2pt) $) --
		      ($ (a40.south east) + ( 2pt,-2pt) $) --
		      ($ (a40.north east) + ( 2pt, 2pt) $) -- 
		      ($ (a30.north west) + (-2pt, 2pt) $) -- cycle;
		\path[
		      draw,
		      rounded corners, 
		      line width=0.6pt
		    ] 
		      ($ (a10.south west) + (-2pt,-2pt) $) --
		      ($ (a20.south east) + ( 2pt,-2pt) $) --
		      ($ (a20.north east) + ( 2pt, 2pt) $) -- 
		      ($ (a10.north west) + (-2pt, 2pt) $) -- cycle;
		\path[
		      draw,
		      rounded corners, 
		      line width=0.6pt
		    ] 
		      ($ (a00.south west) + (-2pt,-2pt) $) --
		      ($ (a00.south east) + ( 2pt,-2pt) $) --
		      ($ (a00.north east) + ( 2pt, 2pt) $) -- 
		      ($ (a00.north west) + (-2pt, 2pt) $) -- cycle;
	\end{tikzpicture}\vspace{0pt}
\]So the resulting complex is a direct sum of subcomplexes indexed by the set of partitions $\lambda$ with at most $k$ parts and $|\lambda|_\infty \leq m$. The subcomplex corresponding to $\lambda = (\lambda_1,\ldots,\lambda_r,0,\ldots,0)$ with $\lambda_r > 0$ consists of the objects $U(\mu)$ for partitions $\mu = (\mu_1,\ldots,\mu_r,0,\ldots,0)$ satisfying $2\lambda_i - 1 \leq \mu_i \leq 2\lambda_i$ for $i \in\{1,\ldots,r\}$. We then show that this subcomplex corresponding to $\lambda$ is $\mathrm{F}_\lambda \otimes_{H^*(\BU(N))} \Z$ up to a grading shift that we compute. 

\begin{proof}[Proof of Proposition~\ref{prop:slNcomplexT2m}]
	For notational simplicity, let us first assume that $k \leq N-k$. We consider the twist closure $\ol{\sr{F}^{2m}(\sr{P}_k)T} = \smash{\bigoplus_{|\mu|_\infty \leq 2m} t^{-|\mu|}\,\ol{W(\mu)T}}$. First, the twist closure $\ol{W(\mu)T}$ of $W(\mu)$ is a complex that deformation retracts onto a grading-shifted web $U(\mu)$ supported in a single $t$-grading, given by \[
		U(\mu)\coloneqq t^{-r(\mu)}q^{H(\mu) + r(\mu)(N+1)-k(N-k)} \quad \begin{gathered}
			\labellist
			\pinlabel $r_1$ at 1.4 7
			\pinlabel $r_2$ at 2.9 7
			\pinlabel $\cdots$ at 4.6 7
			\pinlabel $r_m$ at 5.8 7
			\pinlabel $r$ at 12.8 7
			\pinlabel $k-r$ at 9.2 8.3
			\pinlabel $k$ at 10.3 5.5
			\pinlabel $k$ at 0.9 0.9
			\pinlabel $r$ at 8.1 1.8
			\pinlabel $k+r$ at 10.1 0.9
			\endlabellist
			\includegraphics[width=.35\textwidth]{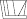}
		\end{gathered}
	\]
	The proof is just the proof of \cite[Lemma 3.4]{MR4903257} after applying standard associator foams. By the homological perturbation lemma (see \cite[section 3.1]{MR4903257}), these individual deformation retracts $\ol{W(\mu)T} \to U(\mu)$ determine a deformation retract of $\ol{\sr{F}^{2m}(\sr{P}_k)T}$ onto a complex with objects $\smash{\bigoplus_{|\mu|_\infty \leq 2m}} t^{-|\mu|}\,U(\mu)$. Because the differential of the original complex $\sr{F}^{2m}(\sr{P}_k)$ respects the partial ordering on the indexing set of partitions given by entrywise comparison, the retract complex does as well. If $\mu$ and $\mu'$ are partitions such that $\mu' < \mu$ entrywise and for which $t^{-|\mu|}U(\mu)$ and $t^{-|\mu'|}U(\mu')$ differ in $t$-degree by $1$, we see that $\mu'$ must be obtainable from $\mu$ by decreasing a single entry of $\mu$ by $1$ and $r(\mu) = r(\mu')$. The component map $U(\mu) \to U(\mu')$ in this case is the unique map for which \vspace{-5pt}\[
		\begin{tikzcd}[row sep=10pt]
			\ol{W(\mu')T} \ar[d] & \ol{W(\mu)T} \ar[l] \ar[d]\\
			U(\mu') & U(\mu) \ar[l]
		\end{tikzcd}\vspace{-5pt}
	\]commutes, where the vertical maps are the deformation retracts and the top map is induced by the component of the differential from $W(\mu)$ to $W(\mu')$ in $\sr{F}^{2m}(\sr{P}_k)$. All other component maps are zero because the differential has $t$-degree $1$.

	Suppose $\mu'$ is obtained by decreasing the $i$th entry $\mu_i$ of $\mu$ by $1$ and $r(\mu') = r(\mu)$. We claim that if $\mu_i$ is odd, then the component of the differential from $U(\mu)$ to $U(\mu')$ is zero. The assumption that $r(\mu) = r(\mu')$ implies that $\mu_i$ is an odd number greater than $1$. Observe that the component of the differential of $\sr{F}^{2m}(\sr{P}_k)$ from $W(\mu)$ to $W(\mu')$ is a composite that factors through a grading shift of the web $V_{r(\mu)-1}$. It follows that the component map $\ol{W(\mu)T} \to \ol{W(\mu')T}$ factors through $\ol{V_{r(\mu)-1}T}$, which deformation retracts onto a shifted web supported in $t$-degree $t^{-r(\mu)+1}$. As $U(\mu)$ and $U(\mu')$ are supported in $t$-degree $t^{-r(\mu)}$, it follows that the component map $U(\mu) \to U(\mu')$ is zero.

	So we have homotopy equivalences \[
		\KRC_{k,N}(T(2,2m+1)) \simeq q^{k(N-k)(2m+1)}\, \ol{\sr{F}^{2m}(\sr{P}_k)T} \simeq q^{k(N-k)(2m-1)}\, \bigoplus_{|\mu|_\infty\leq 2m} t^{-|\mu|}q^{2k(N-k)}U(\mu)
	\]where the nonzero components of the differential of the right-hand complex correspond to decreasing even entries of $\mu$ by $1$. For each partition $\lambda = (\lambda_1,\ldots,\lambda_r,0,\ldots,0)$ with $m \ge \lambda_1$ and $\lambda_r > 0$, let $M(\lambda)$ be the set of partitions $\mu = (\mu_1,\ldots,\mu_r,0,\ldots,0)$ for which $2\lambda_i-1 \leq \mu_i \leq 2\lambda_i$ for $i\in\{1,\ldots,r\}$. Let $\mu(\lambda) \in M(\lambda)$ be the partition $\mu(\lambda) \coloneqq (2\lambda_1-1,\ldots,2\lambda_r-1,0,\ldots,0)$. Then the right-hand complex above splits as a direct sum of complexes \[
		\KRC_{k,N}(T(2,2m+1)) \simeq q^{k(N-k)(2m-1)}\, \bigoplus_{|\lambda|_\infty\leq m} \mathrm{C}_\lambda \qquad\qquad \mathrm{C}_\lambda \coloneqq \bigoplus_{\mu\in M(\lambda)} t^{-|\mu|}q^{2k(N-k)}U(\mu).
	\]
	We identify $\mathrm{C}_\lambda$ with $\mathrm{D}_\lambda \coloneqq t^{-2|\lambda|}q^{4\sum_i i\lambda_i}\,\mathrm{F}_\lambda \otimes_{H^*(\BU(N))} \Z$, which appears in the statement of the proposition. Recall that $\mathrm{F}_\lambda = B^{r_1,\ldots,r_m} \otimes \Sym(\mathbf{Z}) \otimes \Sym(\mathbf{W})$ where $B^{r_1,\ldots,r_m}$ is the direct sum $\bigoplus_{v} B^{r_1,\ldots,r_m}(v)$ over sequences $v\in\{0,1\}^r$ of the form $(1^{a_1},0^{r_1-a_1},1^{a_2},0^{r_2-a_2},\ldots,1^{a_m},0^{r_m-a_m})$ for integers $a_i$ satisfying $0 \leq a_i \leq r_i$. Note that $v$ is of this form if and only if $\mu(\lambda) + v \coloneqq \mu(\lambda) + (v_1,\ldots,v_r,0^{k-r})$ is a partition. In particular, this defines a bijection between the indexing set of vertices $v$ defining $\mathrm{D}_\lambda$ and the index set $M(\lambda)$ of partitions $\mu$ defining $\mathrm{C}_\lambda$. For each such vertex $v$, we identify the summand of $\mathrm{D}_\lambda$ corresponding to $v$ with the summand of $\mathrm{C}_\lambda$ corresponding to $\mu(\lambda) + v$. 

	The summand of $\mathrm{D}_\lambda$ indexed by $v$ is \[
		t^{-2|\lambda| - |v|}q^{4\sum_i i\lambda_i + 2\sum_i iv_i} \,\left(\Z[\mathbf{X}]^{\fk{S}_{v}} \otimes \Sym(\mathbf{Y}) \otimes \Sym(\mathbf{Z}) \otimes \Sym(\mathbf{W})\right) \otimes_{H^*(\BU(N))} \Z
	\]where $\fk{S}_v \coloneqq \fk{S}_{a_1} \x \fk{S}_{r_1-a_1} \x \cdots \x \fk{S}_{a_m} \x \fk{S}_{r_m-a_m}$. As before, the sizes of $\mathbf{X},\mathbf{Y},\mathbf{Z},\mathbf{W}$ are $r,r,k-r,N-k-r$, respectively, and $H^*(\BU(N)) = \Sym(\mathbf{X} \sqcup\mathbf{Y} \sqcup\mathbf{Z}\sqcup\mathbf{W})$. The corresponding summand of $\mathrm{C}_\lambda$ is $t^{-|\mu(\lambda) + v|}q^{2k(N-k)}U(\mu(\lambda) + v)$ where \[
		U(\mu(\lambda) + v) = t^{-r(\mu(\lambda)+v)}q^{H(\mu(\lambda) + v) + r(\mu(\lambda)+v)(N+1)-k(N-k)} \qquad \begin{gathered}
			\labellist
			\pinlabel $\mathbf{X}_1^*$ at 1.2 7
			\pinlabel $\mathbf{X}_2^*$ at 2.7 7
			\pinlabel $\cdots$ at 4.3 7
			\pinlabel $\mathbf{X}_{2m}^*$ at 5.6 7
			\pinlabel $\mathbf{Y}$ at 12.9 7
			\pinlabel $\mathbf{Z}$ at 9.2 8.3
			\pinlabel ${\mathbf{Y}\!\sqcup\!\mathbf{Z}}$ at 9.8 5.5
			\pinlabel ${\mathbf{X}\!\sqcup\!\mathbf{Z}}$ at 0.3 0.9
			\pinlabel $\mathbf{X}$ at 8.1 1.8
			\pinlabel ${\mathbf{X}\!\sqcup\!\mathbf{Y}\!\sqcup\!\mathbf{Z}}$ at 10.1 0.9
			\endlabellist
			\includegraphics[width=.35\textwidth]{twistClosedWeb}
		\end{gathered}
	\]where each edge is given a set of indeterminates of size equal to the label of the edge. The sizes of $\mathbf{X}_1^*,\ldots,\mathbf{X}_{2m}^*$ are $a_1,r_1-a_1,\ldots,a_m,r_m-a_m$, whose nonzero terms are the multiplicities of the parts of $\mu(\lambda) + v$, and $\mathbf{X} = \mathbf{X}_1^*\sqcup\cdots\sqcup\mathbf{X}^*_{2m}$. We identify $\Z[\mathbf{X}]^{\fk{S}_v} = \Sym(\mathbf{X}_1^*) \otimes\cdots\otimes\Sym(\mathbf{X}_{2m}^*)$. By the same argument used in the proof of \cite[Proposition 5.2]{MR4903257}, the $\sl(N)$ state space of the unshifted web underlying $U(\mu(\lambda) + v)$ is \[
		q^{(a_1^2+(r_1-a_1)^2 + \cdots+a_m^2 +(r_m-a_m)^2+r^2+(k-r)^2+(N-k-r)^2 - N^2)/2}\left(\Z[\mathbf{X}]^{\fk{S}_v} \otimes \Sym(\mathbf{Y}) \otimes \Sym(\mathbf{Z}) \otimes \Sym(\mathbf{W}) \right) \otimes_{H^*(\BU(N))} \Z.
	\]Note that this is the Borel presentation of the cohomology ring of $\mathrm{Fl}(a_1,r_1-a_1,\ldots,a_m,r_m-a_m,r,k-r,N-k-r;\C^N)$ shifted downwards in grading by its complex dimension. Verifying that these summands of $\mathrm{D}_\lambda$ and $\mathrm{C}_\lambda$ match reduces to checking that the grading shifts are the same. The $t$-grading shifts agree using the observation that $r(\mu(\lambda) + v) = r$. For the $q$-grading shifts, we must show that $4\sum_i i\lambda_i + 2\sum_i iv_i$ is equal to \[
		2k(N-k) + H(\mu(\lambda) + v) + r(N+1)-k(N-k)+\frac12\left(\sum_{i=1}^m (a_i^2 + (r_i-a_i)^2)\right) + \frac12(r^2 + (k-r)^2 + (N-k-r)^2 -N^2).
	\]Direct simplification using the formula $H(\mu) = \left(\sum_{i=1}^r 2i\mu_i \right) - \frac12(r^2 + \sum_{j=1}^m r_j^2)$ verifies this equality. 

	Lastly, we show that the differentials of $\mathrm{C}_\lambda$ and $\mathrm{D}_\lambda$ agree. By construction, the complex $\mathrm{C}_\lambda$ is a deformation retract of the subquotient complex $\bigoplus_{\mu\in M(\lambda)} t^{-|\mu|}\,\ol{W(\mu)T}$ of $\ol{\sr{F}^{2m}(\sr{P}_k)T}$ up to a grading shift. Consider the subquotient complex $\bigoplus_{\mu\in M(\lambda)} t^{-|\mu|}\,W(\mu)$ of $\sr{F}^{2m}(\sr{P}_k)$. By definition of the differential of $\sr{F}^{2m}(\sr{P}_k)$, there is a chain map \[
		t^{2|\lambda|}\hspace{-2pt}\bigoplus_{\mu\in M(\lambda)} t^{-|\mu|}\,W(\mu) \to \,t^{|\varepsilon(\mu(\lambda))|}\hspace{-10pt}\bigoplus_{\varepsilon \in \{1,2\}^r \x \{0\}^{k-r}} t^{-|\varepsilon|}\,V(\varepsilon)
	\]given by standard inclusion foams $W(\mu) \to V(\varepsilon(\mu))$. The $t$-degree shifts are chosen so that this map has $t$-degree zero. The complex on the right is a subquotient complex of the auxiliary complex $\sr{K}_k$, and its differential is given explicitly before the statement of the proposition. Note that for each $\varepsilon \in \{1,2\}^r \x \{0\}^{k-r}$, the web $V(\varepsilon)$ is just a grading-shifted copy of $V_r$. Just as for $W(\mu)$, the twist closure of $V_r$ deformation retracts onto a web whose state space is a grading-shifted copy of $(\Z[\mathbf{X}] \otimes \Sym(\mathbf{Y}) \otimes \Z[\mathbf{Z}] \otimes \Sym(\mathbf{W})) \otimes_{H^*(\BU(N))} \Z$. The entire subquotient complex $\bigoplus_{\varepsilon\in\{1,2\}^r\x\{0\}^{k-r}}t^{-|\varepsilon|}\,V(\varepsilon)$ deformation retracts onto a grading-shifted copy of $(C^r \otimes \Z[\mathbf{Z}] \otimes \Sym(\mathbf{W})) \otimes_{H^*(\BU(N))} \Z$, where $C^r$ is the complex defined in Definition~\ref{def:complexC}. This follows directly from the explicit formula for the differential of the subquotient complex and the dot sliding relation (see for example \cite[Proposition 2.18]{MR4903257}). It follows that the complex $\mathrm{C}_\lambda$ maps injectively into $(C^r \otimes \Z[\mathbf{Z}] \otimes \Sym(\mathbf{W})) \otimes_{H^*(\BU(N))} \Z$ with image $(B^{r_1,\ldots,r_m} \otimes \Sym(\mathbf{Z}) \otimes \Sym(\mathbf{W})) \otimes_{H^*(\BU(N))} \Z$, which, up to a grading shift, is precisely $\mathrm{D}_\lambda$. This completes the proof in the case of $T(2,2m+1)$ when $k \leq N-k$. If $k > N-k$, we note that any web with an edge with a label larger than $N$ is zero in the $\sl(N)$ webs and foams category, so $W(\mu) = 0$ whenever $\mu$ has more than $N-k$ parts. The case of $T(2,\infty)$ follows in the same way, except without the restrictions $|\mu|_\infty \leq 2m$ and $|\lambda|_\infty \leq m$, and a straightforward global grading-shift calculation. 
\end{proof}

\begin{proof}[Proof of Theorem~\ref{thm:mainTheorem}]
	This follows directly from Corollary~\ref{cor:freeloopspaceshifts}, Proposition~\ref{prop:Flambdafreeresolution}, and Proposition~\ref{prop:slNcomplexT2m}.
\end{proof}

\newpage
\appendix
\section{The nil-Hecke algebra action}\label{sec:nil_hecke_action}

The purpose of this appendix is to construct an action of the nil-Hecke algebra $\cl{H}_r$ on the chain complex $C^r$ by chain maps. The arguments in this appendix are routine but lengthy. We recall that $\cl{H}_r$ is generated as an algebra by $x_1,\ldots,x_r,\partial_1,\ldots,\partial_{r-1}$ with relations \[
	\begin{gathered}
		x_ix_j = x_jx_i \qquad \partial_i\partial_i = 0 \qquad \partial_i \partial_{j}\partial_i = \partial_{j}\partial_i\partial_{j}\text{ for }|i - j| = 1 \qquad \partial_i\partial_j = \partial_j \partial_i \text{ for } |i - j| > 1\\
		\partial_ix_i = \Id + \,x_{i+1}\partial_i \qquad \partial_ix_{i+1} = -\Id + \,x_i\partial_i \qquad \partial_ix_j = x_j \partial_i \text{ for }j \notin \{i,i+1\}.
	\end{gathered}
\]The simple transpositions $s_1,\ldots,s_{r-1} \in \cl{H}_r$ generate a copy of $\Z[\fk{S}_r]$ within $\cl{H}_r$, and we have \[
	\begin{gathered}
		s_is_i = \Id \qquad s_is_js_i = s_js_is_j \text{ for }|i-j| = 1 \qquad s_is_j = s_js_i \text{ for }|i-j| > 1\\
		s_ix_{i+1} = x_is_i \qquad s_ix_i = x_{i+1}s_i \qquad s_ix_j = x_js_i\text{ for }j \notin \{i,i+1\}\\
		s_i\partial_i=\partial_i \qquad \partial_is_i = -\partial_i \qquad s_is_j\partial_i = \partial_j s_i s_j \text{ and } s_i\partial_js_i = s_j\partial_is_j \text{ for } |i-j|=1 \qquad s_i\partial_j = \partial_j s_i \text{ for }|i-j|>1.
	\end{gathered}
\]The image and kernel of $\partial_i$ agree and coincide with the set of polynomials invariant under $s_i$. We refer to the relations above with either $j\notin \{i,i+1\}$ or $|i-j| > 1$ as \emph{far-commutativity relations} and those with $|i-j| = 1$ as \emph{braid relations}. 
The complex $C^r$ is defined in Definition~\ref{def:complexC}. The endomorphisms $X_1,\ldots,X_r,\Delta_1,\ldots,\Delta_{r-1}$ of $C^r$ are defined in Definition~\ref{def:nilHeckeActionOnC}. Lemmas~\ref{lem:DeltaChainMap} and \ref{lem:XChainMap} verify that these endomorphisms are chain maps. The nil-Hecke relations are proved in Lemmas~\ref{lem:XsCommute}, \ref{lem:nilCoxeterRelations}, \ref{lem:twistedLeibnizRule}, and \ref{lem:farCommutativity}. 

\begin{lem}\label{lem:DeltaChainMap}
	The endomorphisms $\Delta_1,\ldots,\Delta_{r-1}$ are chain maps.
\end{lem}
\begin{proof}
	We verify that $\Delta_i$ is a chain map. Fix $v_1,\ldots,v_{i-1},v_{i+2},\ldots,v_r\in\{0,1\}$. For $a,b\in\{0,1\}$, let $v^{a,b}\in\{0,1\}^r$ be the vertex whose $i$th and $(i+1)$th entries are $a$ and $b$, respectively, and whose other entries are $v_1,\ldots,v_{i-1},v_{i+2},\ldots,v_r$. Let $D = \theta_{i+1}\cdots\theta_{r-1}\,Q_r\,\hat{\theta}_{r-1}\cdots\hat{\theta}_{i+1}$ where $\theta_{i+1},\ldots,\theta_{r-1}$ are defined in terms of $v_{i+2},\ldots,v_r$ as in Definition~\ref{def:complexC}. Then the components of $\Delta_i$ and $d$ between these four vertices are \[
		\hspace{-20pt}\begin{tikzcd}[column sep=40pt,row sep=20pt]
			R(v^{01}) \ar[dr,"\textstyle s_i"] & R(v^{11}) \ar[out=-10,in=10,distance=50pt,loop,swap,"\textstyle\partial_i"]\\
			R(v^{00}) \ar[out=170,in=190,distance=50pt,loop,swap,"\textstyle\partial_i"] & R(v^{10})
		\end{tikzcd}\hspace{40pt}\begin{tikzcd}[column sep=40pt,row sep=20pt]
			R(v^{01}) \ar[d,swap,"\textstyle D"] & R(v^{11}) \ar[l,swap,"\textstyle{s_i\,D\,\partial_i}"] \ar[d,swap,"\textstyle D"] \\
			R(v^{00}) & R(v^{10}) \ar[l,swap,"\textstyle{\partial_i\,D\,s_i}"]
		\end{tikzcd}
	\]
	respectively. While there are other components of $d$ that map to and from these four vertices, there are no other such components of $\Delta_i$. The component maps of $[d,\Delta_i]$ between these four vertices may therefore be computed using only the displayed components of $d$. The verification that these component maps of $[d,\Delta_i]$ are zero is straightforward. 

	Now suppose $j$ satisfies $i+2 \leq j \leq r$ and $v_j = 1$. Let $w^{a,b} \in \{0,1\}^r$ be the vertex obtained from $v^{a,b}$ by changing its $j$th entry from $1$ to $0$. The component map of $d$ from $R(v^{a,b})$ to $R(w^{a,b})$ is $\theta_j\cdots\theta_{r-1}\,Q_r\,\hat{\theta}_{r-1}\cdots\hat{\theta}_j$ which commutes with $s_i$ and $\partial_i$ because all subscripts that appear are strictly greater than $i+1$. It follows from this observation that the component map of $[d,\Delta_i]$ from any vertex among $v^{00},v^{10},v^{01},v^{11}$ to any vertex among $w^{00},w^{10},w^{01},w^{11}$ is zero.

	Lastly, suppose $j$ satisfies $1 \leq j \leq i-1$ and $v_j = 1$. Let $u^{a,b}$ be obtained from $v^{a,b}$ by changing its $j$th entry from $1$ to $0$. The component of $d$ from $R(v^{11})$ to $R(u^{11})$ is $\theta_j\cdots\theta_{i-2}\,s_{i-1}\,s_{i}\,D\,\partial_i\,\partial_{i-1}\,\hat{\theta}_{i-2}\cdots\hat{\theta}_j$ which commutes with $\partial_i$ because \begin{align*}
		&\theta_j\cdots\theta_{i-2}\,s_{i-1}\,s_{i}\,D\,\partial_i\,\partial_{i-1}\,\hat{\theta}_{i-2}\cdots\hat{\theta}_j\,\boxed{\partial_i}\\
		=\:\: &\theta_j\cdots\theta_{i-2}\,s_{i-1}\,s_{i}\,D\,\partial_i\,\partial_{i-1}\,\boxed{\partial_i}\,\hat{\theta}_{i-2}\cdots\hat{\theta}_j & \text{(far commutativity)}\\
		=\:\: &\theta_j\cdots\theta_{i-2}\,s_{i-1}\,s_{i}\,D\,\boxed{\partial_{i-1}}\,\partial_i\,\partial_{i-1}\,\hat{\theta}_{i-2}\cdots\hat{\theta}_j & \text{(braid relation)}\\
		=\:\: &\theta_j\cdots\theta_{i-2}\,s_{i-1}\,s_{i}\,\boxed{\partial_{i-1}}\,D\,\partial_i\,\partial_{i-1}\,\hat{\theta}_{i-2}\cdots\hat{\theta}_j & \text{(far commutativity)}\\
		=\:\: &\theta_j\cdots\theta_{i-2}\,\boxed{\partial_{i}}\,s_{i-1}\,s_{i}\,D\,\partial_i\,\partial_{i-1}\,\hat{\theta}_{i-2}\cdots\hat{\theta}_j & \text{(braid relation)}\\
		=\:\: &\boxed{\partial_{i}}\,\theta_j\cdots\theta_{i-2}\,s_{i-1}\,s_{i}\,D\,\partial_i\,\partial_{i-1}\,\hat{\theta}_{i-2}\cdots\hat{\theta}_j & \text{(far commutativity)}
	\end{align*}
	The component of $[d,\Delta_i]$ from $R(v^{11})$ to $R(u^{11})$ is therefore zero. Similar computations verify that the components of $[d,\Delta_i]$ from $R(v^{00})$ to $R(u^{00})$ and from $R(v^{01})$ to $R(u^{10})$ are zero. This completes the proof that $\Delta_i$ is a chain map.
\end{proof}

\begin{lem}\label{lem:XChainMap}
	The endomorphisms $X_1,\ldots,X_r$ are chain maps.
\end{lem}
\begin{proof}
	We verify that $X_i = x_i + \xi_i$ is a chain map. It suffices to show that $[d,x_i] = [\xi_i,d]$. Fix a vertex $v\in\{0,1\}^r$ and an index $j \in \{1,\ldots,r\}$ for which $v_j = 1$. Let $w \in \{0,1\}^r$ be obtained from $v$ by changing its $j$th entry from $1$ to $0$. We show that the components of $[d,x_i]$ and $[\xi_i,d]$ from $R(v)$ to $R(w)$ are equal using the following cases. Let us first consider the case that $v_i = 1$. \begin{itemize}[noitemsep]
		\item[$\bullet$] Case $i < j$. Then $[d,x_i]\colon R(v) \to R(w)$ is zero by far commutativity. We show that $d\,\xi_i\colon R(v) \to R(w)$ is zero. Suppose $u$ is a vertex for which $\xi_i\colon R(v) \to R(u)$ is nonzero. Then $u_i$ must be $0$, so $d\colon R(u) \to R(w)$ is zero as $w_i = 1$. Similarly, if $u$ is a vertex for which $\xi_i\colon R(u) \to R(w)$ is nonzero, then there is an index $l < i$ such that $u_l = 1$ while $w_l = v_l = 0$, so $d\colon R(v) \to R(u)$ must be zero. 
		\item[$\bullet$] Case $i = j$. Then $d\colon R(v) \to R(w)$ is $\theta_{i}\,\cdots\,\theta_{r-1}\,Q_r\,\hat{\theta}_{r-1}\,\cdots\,\hat{\theta}_{i}$. Observe that 
		$\hat{\theta}_l\,x_l = x_{l+1}\,\hat{\theta}_l$ if $v_{l+1} = 0$ while $\hat{\theta}_l\,x_l = x_{l+1}\,\hat{\theta}_l + \Id$ if $v_{l+1} = 1$, so 
		\[
			\theta_{i}\cdots\theta_{r-1}\,Q_r\,\hat{\theta}_{r-1}\cdots\hat{\theta}_{i}\,\boxed{x_i} = \theta_{i}\cdots\theta_{r-1}\,Q_r\,\boxed{x_r}\,\hat{\theta}_{r-1}\cdots\hat{\theta}_{i} + \sum_{\substack{i \leq l \leq r-1\\v_{l+1} = 1}} \theta_{i}\cdots\theta_{r-1}\,Q_r\,\hat{\theta}_{r-1}\cdots\hat{\theta}_{l+1}\,\hat{\theta}_{l-1}\cdots\hat{\theta}_{i}.
		\]Note that the term in the sum indexed by $l$ is \[
			\theta_{i}\cdots\theta_{l-1}\,s_l\,(\theta_{l+1}\cdots\theta_{r-1}\,Q_r\,\hat{\theta}_{r-1}\cdots\hat{\theta}_{l+1})\,\hat{\theta}_{l-1}\cdots\hat{\theta}_{i} = \theta_{i}\cdots\theta_{l-1}\,s_l\,\hat{\theta}_{l-1}\cdots\hat{\theta}_{i}\,(\theta_{l+1}\cdots\theta_{r-1}\,Q_r\,\hat{\theta}_{r-1}\cdots\hat{\theta}_{l+1})
		\]by far commutativity. If $u$ is the vertex obtained from $v$ by changing its $(l+1)$th entry from $1$ to $0$, then this expression is precisely the composite of $d\colon R(v) \to R(u)$ followed by $\xi_i\colon R(u) \to R(w)$. The sum over all such $l$ is exactly $\xi_i\,d\colon R(v)\to R(w)$. Continuing onwards using $\theta_l\,x_{l+1} = x_l\,\theta_l-\Id$ if $v_{l+1} = 0$ and $\theta_l\,x_{l+1} = x_l\,\theta_l$ if $v_{l+1} = 1$, we obtain \[
			\theta_{i}\cdots\theta_{r-1}\,\boxed{x_r}\,Q_r\,\hat{\theta}_{r-1}\cdots\hat{\theta}_{i} = \boxed{x_i}\,\theta_i\cdots\theta_{r-1}\,Q_r\,\hat{\theta}_{r-1}\cdots\hat{\theta}_i - \sum_{\substack{i\leq l \leq r-1\\v_{l+1}=0}} \theta_i\cdots\theta_{l-1}\,\theta_{l+1}\cdots\theta_{r-1}\,Q_r\,\hat{\theta}_{r-1}\cdots\hat{\theta}_i.
		\]The term in the sum indexed by $l$ is \[
			\theta_i\cdots\theta_{l-1}\,(\theta_{l+1}\cdots\theta_{r-1}\,Q_r\,\hat{\theta}_{r-1}\cdots\hat{\theta}_{l+1})\,s_l\,\hat{\theta}_{l-1}\cdots\hat{\theta}_i = (\theta_{l+1}\cdots\theta_{r-1}\,Q_r\,\hat{\theta}_{r-1}\cdots\hat{\theta}_{l+1})\,\theta_i\cdots\theta_{l-1}\,s_l\,\hat{\theta}_{l-1}\cdots\hat{\theta}_i
		\]again by far commutativity. If $u$ is the vertex obtained from $v$ by swapping its $i$th and $(l+1)$th entries, then this expression is precisely the composite of $\xi_i\colon R(v) \to R(u)$ followed by $d\colon R(u) \to R(w)$, and the sum over all such $l$ is exactly $d\,\xi_i \colon R(v) \to R(w)$. Thus $[d,x_i] = [\xi_i,d]$ in this case.
		\item[$\bullet$] Case $j < i$. Then $d \colon R(v) \to R(w)$ is $\theta_j\cdots\theta_{i-2}\,s_{i-1}\,\theta_{i}\cdots\theta_{r-1}\,Q_r\,\hat{\theta}_{r-1}\cdots\hat{\theta}_{i}\,\partial_{i-1}\,\hat{\theta}_{i-2}\cdots\hat{\theta}_j$, whose commutator with $x_i$ is \[
			- \,\theta_j\cdots\theta_{i-2}\,s_{i-1}\,(\theta_{i}\cdots\theta_{r-1}\,Q_r\,\hat{\theta}_{r-1}\cdots\hat{\theta}_{i})\,\hat{\theta}_{i-2}\cdots\hat{\theta}_j = - \,\theta_j\cdots\theta_{i-2}\,s_{i-1}\,\hat{\theta}_{i-2}\cdots\hat{\theta}_j\,(\theta_{i}\cdots\theta_{r-1}\,Q_r\,\hat{\theta}_{r-1}\cdots\hat{\theta}_{i})
		\]Let $u$ be the vertex obtained from $v$ by changing its $i$th entry from $1$ to $0$. This expression, sign included, is precisely $d\colon R(v) \to R(u)$ followed by $\xi_i\colon R(u) \to R(w)$. By analysis similar to the first case above, there is no other vertex $u'$ such that both $d\colon R(v) \to R(u')$ and $\xi_i\colon R(u') \to R(w)$ are nonzero, nor is there a vertex $u'$ such that both $\xi_i\colon R(v) \to R(u')$ and $d\colon R(u')\to R(w)$ are nonzero. Hence $[d,x_i] = [\xi_i,d]$ in this case as well.
	\end{itemize}Next, let us consider the case that $v_i = 0$.
	\begin{itemize}[noitemsep]
		\item[$\bullet$] Case $i < j$. The components of $[d,x_i]$, $d\,\xi_i$, and $\xi_i\,d$ from $R(v)$ to $R(w)$ are straightforwardly zero. 
		\item[$\bullet$] Case $j < i$. Then $d\colon R(v) \to R(w)$ is $\theta_j\cdots\theta_{i-2}\,\partial_{i-1}\,\theta_{i}\cdots\theta_{r-1}\,Q_r\,\hat{\theta}_{r-1}\cdots\hat{\theta}_{i}\,s_{i-1}\,\hat{\theta}_{i-2}\cdots\hat{\theta}_j$, whose commutator with $x_i$ is \[
			- \,\theta_j\cdots\theta_{i-2}\,(\theta_{i}\cdots\theta_{r-1}\,Q_r\,\hat{\theta}_{r-1}\cdots\hat{\theta}_{i})\,s_{i-1}\,\hat{\theta}_{i-2}\cdots\hat{\theta}_j = - \,(\theta_{i}\cdots\theta_{r-1}\,Q_r\,\hat{\theta}_{r-1}\cdots\hat{\theta}_{i})\,\theta_j\cdots\theta_{i-2}\,s_{i-1}\,\hat{\theta}_{i-2}\cdots\hat{\theta}_j.
		\]This expression is $\xi_i\colon R(v) \to R(u)$ followed by $d\colon R(u) \to R(w)$ where $u$ is obtained from $v$ by swapping its $i$th and $j$th entries. There are no other contributions to $[\xi_i,d]\colon R(v) \to R(w)$ just as in the third case above, which completes the argument. 
	\end{itemize}We have verified that $[d,x_i] = [\xi_i,d]$ along any edge of the cube $[0,1]^r$. Clearly all other components of $[d,x_i]$ are zero, so we must verify that all other components of $[\xi_i,d]$ are zero as well. 

	Suppose we have a vertex $v$ and indices $j,l$ such that $i,j,l$ are all distinct, $v_l = 1$, and $v_i \neq v_j$. Let $w$ be the vertex obtained from $v$ by changing its $l$th entry from $1$ to $0$ and by swapping its $i$th and $j$th entries. Clearly $v$ and $w$ are not endpoints of an edge of the cube. We show that the component of $[\xi_i,d]$ from $R(v)$ to $R(w)$ is zero. This argument will complete the proof because all other components of $[\xi_i,d]$ are either along edges of the cube or are clearly zero. Let $u$ be obtained from $v$ by swapping its $i$th and $j$th entries, and let $u'$ be obtained from $v$ by changing its $l$th entry from $1$ to $0$. We show that \[
		\begin{tikzcd}
			R(u) \ar[d,swap,"\textstyle d"] & R(v) \ar[l,swap,"\textstyle \xi_i"] \ar[d,swap,"\textstyle d"]\\
			R(w) & R(u') \ar[l,swap,"\textstyle \xi_i"]
		\end{tikzcd}
	\]commutes. There are six cases depending on the ordering of $i,j,l$. Let us first consider the case that $i < j$. As the components of $\xi_i$ in the diagram would be zero if $(v_i,v_j) = (0,1)$, we may assume $(v_i,v_j) = (1,0)$. \begin{itemize}[noitemsep]
		\item[$\bullet$] Case $i < j < l$. The diagram commutes by far commutativity.
		\item[$\bullet$] Case $i < l < j$. We first show that $d\,\xi_i$ is zero. Note that $d\,\xi_i$ is given by \[
			(\theta_l\cdots\theta_{j-2}\,s_{j-1}\,\theta_j\cdots\theta_{r-1}\,Q_r\,\hat{\theta}_{r-1}\cdots\hat{\theta}_j\,\partial_{j-1}\,\hat{\theta}_{j-2}\cdots\hat{\theta}_l)\,(\theta_i\cdots\theta_{l-2}\,s_{l-1}\,\theta_l\cdots\theta_{j-2}\,s_{j-1}\,\hat{\theta}_{j-2}\cdots\hat{\theta}_l\,\partial_{l-1}\,\hat{\theta}_{l-2}\cdots\hat{\theta}_i)
		\]so it suffices to show that $\partial_{j-1}\,\hat{\theta}_{j-2}\cdots\hat{\theta}_l\,(s_{l-1}\,\theta_l\cdots\theta_{j-2}\,s_{j-1}\,\hat{\theta}_{j-2}\cdots\hat{\theta}_l\,\partial_{l-1}) = 0$. Let $\theta_{m,-1} \in \{\partial_{m-1},s_{m-1}\}$ be obtained from $\theta_m$ by decreasing its subscript from $m$ to $m-1$. Similarly define $\hat{\theta}_{m,-1}$. Note that $\theta_{m,-1}$ and $\theta_{m-1}$ do not have to agree, since they are defined in terms of $v_{m+1}$ and $v_m$. The key identities are $\hat{\theta}_m\,s_{m-1}\,\theta_m = \theta_{m,-1}\,s_m\,\hat{\theta}_{m,-1}$ and $\hat{\theta}_{m,-1}\,\hat{\theta}_m\,\partial_{m-1} = \partial_m\,\hat{\theta}_{m,-1}\,\hat{\theta}_m$, which follow from the braid relations. We obtain \begin{align*}
			&\quad\: \partial_{j-1}\,\boxed{\hat{\theta}_{j-2}\cdots\hat{\theta}_l\,s_{l-1}\,\theta_l\cdots\theta_{j-2}}\,s_{j-1}\,\hat{\theta}_{j-2}\cdots\hat{\theta}_l\,\partial_{l-1}\\
			&= \partial_{j-1}\,\theta_{l,-1}\cdots\theta_{j-2,-1}\,s_{j-2}\,\boxed{\hat{\theta}_{j-2,-1}\cdots\hat{\theta}_{l,-1}}\,s_{j-1}\,\hat{\theta}_{j-2}\cdots\hat{\theta}_l\,\partial_{l-1} & \text{(braid relation)}\\
			&= \partial_{j-1}\,\theta_{l,-1}\cdots\theta_{j-2,-1}\,\boxed{s_{j-2}\,s_{j-1}\,\hat{\theta}_{j-2,-1}\,\hat{\theta}_{j-2}\cdots\hat{\theta}_{l,-1}\,\hat{\theta}_l\,\partial_{l-1}} & \text{(far commutativity)}\\
			&= \boxed{\partial_{j-1}}\,\theta_{l,-1}\cdots\theta_{j-2,-1}\,\boxed{\partial_{j-1}}\,s_{j-2}\,s_{j-1}\,\hat{\theta}_{j-2,-1}\,\hat{\theta}_{j-2}\cdots\hat{\theta}_{l,-1}\,\hat{\theta}_l & \text{(braid relation)}
		\end{align*}which equals zero by far commutativity and the relation $\partial_{j-1}\,\partial_{j-1} = 0$. Next, we claim that $\xi_i\,d$ is also zero. Since $\xi_i\,d$ is \[
			(\theta_i\cdots\theta_{l-2}\,\partial_{l-1}\,\theta_l\cdots\theta_{j-2}\,s_{j-1}\,\hat{\theta}_{j-2}\cdots\hat{\theta}_l\,s_{l-1}\,\hat{\theta}_{l-2}\cdots\hat{\theta}_i)(\theta_l\cdots\theta_{j-2}\,\partial_{j-1}\,\theta_j\cdots\theta_{r-1}\,Q_r\,\hat{\theta}_{r-1}\cdots\hat{\theta}_j\,s_{j-1}\,\hat{\theta}_{j-2}\cdots\hat{\theta}_l)
		\]it suffices to show that $(\partial_{l-1}\,\theta_l\cdots\theta_{j-2}\,s_{j-1}\,\hat{\theta}_{j-2}\cdots\hat{\theta}_l\,s_{l-1})\theta_l\cdots\theta_{j-2}\,\partial_{j-1} = 0$. This follows by the same reasoning. 
		\item[$\bullet$] Case $l < i < j$. Far commutativity reduces the problem to the case that $l = i - 1$. Using the same notation and reasoning as in the previous case, we have \begin{align*}
			d\,\xi_i &= \partial_{i-1}\,\theta_i\cdots\theta_{j-2}\, s_{j-1}\,\theta_j\cdots \theta_{r-1}\,Q_r\,\hat{\theta}_{r-1} \cdots \hat{\theta}_{j}\,\partial_{j-1}\,\boxed{\hat{\theta}_{j-2}\cdots \hat{\theta}_{i}\,s_{i-1}\,\theta_i\cdots \theta_{j-2}}\,s_{j-1}\,\hat{\theta}_{j-2}\cdots\hat{\theta}_i\\
			&= \partial_{i-1}\,\theta_i\cdots\theta_{j-2}\, s_{j-1}\,\theta_j\cdots \theta_{r-1}\,Q_r\,\hat{\theta}_{r-1} \cdots \hat{\theta}_{j}\,\partial_{j-1}\,\boxed{\theta_{i,-1}\cdots\theta_{j-2,-1}}\,s_{j-2}\,\boxed{\hat{\theta}_{j-2,-1}\cdots\hat{\theta}_{i,-1}} \,s_{j-1}\,\hat{\theta}_{j-2}\cdots\hat{\theta}_i\\
			&= \boxed{\partial_{i-1}}\,\theta_i\,\theta_{i,-1}\cdots\theta_{j-2}\,\theta_{j-2,-1}\, s_{j-1}\,\theta_j\cdots \theta_{r-1}\,Q_r\,\hat{\theta}_{r-1} \cdots \hat{\theta}_{j}\,\boxed{\partial_{j-1}}\,s_{j-2}\,s_{j-1}\,\hat{\theta}_{j-2,-1}\,\hat{\theta}_{j-2}\cdots\hat{\theta}_{i,-1}\,\hat{\theta}_i\\
			&= \theta_i\,\theta_{i,-1}\cdots\theta_{j-2}\,\theta_{j-2,-1}\,\boxed{\partial_{j-2}}\,s_{j-1}\,\theta_j\cdots \theta_{r-1}\,Q_r\,\hat{\theta}_{r-1} \cdots \hat{\theta}_{j}\,\boxed{s_{j-2}}\,s_{j-1}\,\hat{\theta}_{j-2,-1}\,\hat{\theta}_{j-2}\cdots\hat{\theta}_{i,-1}\,\hat{\theta}_i \,\partial_{i-1}\\
			&= \theta_i\,\boxed{\theta_{i,-1}}\cdots\theta_{j-2}\,\boxed{\theta_{j-2,-1}}\,s_{j-1}\,s_{j-2}\,\partial_{j-1}\,\theta_j\cdots \theta_{r-1}\,Q_r\,\hat{\theta}_{r-1} \cdots \hat{\theta}_{j}\,s_{j-1}\,\boxed{\hat{\theta}_{j-2,-1}}\,\hat{\theta}_{j-2}\cdots\boxed{\hat{\theta}_{i,-1}}\,\hat{\theta}_i \,\partial_{i-1}\\
			&= \theta_i\cdots\theta_{j-2}\,s_{j-1}\,\boxed{\theta_{i,-1}\cdots\theta_{j-2,-1}\,s_{j-2}\,\hat{\theta}_{j-2,-1}\cdots\,\hat{\theta}_{i,-1}}\,\partial_{j-1}\,\theta_j\cdots \theta_{r-1}\,Q_r\,\hat{\theta}_{r-1} \cdots \hat{\theta}_{j}\,s_{j-1}\,\hat{\theta}_{j-2}\cdots\hat{\theta}_i \,\partial_{i-1}\\
			&= \theta_i\cdots\theta_{j-2}\,s_{j-1}\,\hat{\theta}_{j-2}\cdots\hat{\theta}_i\,s_{i-1}\,\theta_i\cdots\theta_{j-2}\,\partial_{j-1}\,\theta_j\cdots \theta_{r-1}\,Q_r\,\hat{\theta}_{r-1} \cdots \hat{\theta}_{j}\,s_{j-1}\,\hat{\theta}_{j-2}\cdots\hat{\theta}_i \,\partial_{i-1} = \xi_i\,d
		\end{align*}
	\end{itemize}
	This completes the analysis when $i < j$. If $j < i$, then we may assume that $(v_j,v_i) = (1,0)$ as before. By definition of $\xi_i$, the components of $\xi_i$ in the diagram are simply the negation of the corresponding components of $\xi_j$. The previous case applied to $\xi_j$ implies the present case. This completes the proof that $X_i$ is a chain map.
\end{proof}

\begin{lem}\label{lem:XsCommute}
	The maps $X_i$ and $X_j$ commute for any $i,j\in\{1,\ldots,r\}$.
\end{lem}
\begin{proof}
	We show that the component maps of $[\xi_i,x_j] + [x_i,\xi_j] + [\xi_i,\xi_j]$ are zero. First consider the case of a component map $R(v) \to R(w)$ where $v$ and $w$ differ in exactly four entries. In this case, the component maps of $[\xi_i,x_j]$ and $[x_i,\xi_j]$ are zero, so we must show that $[\xi_i,\xi_j] = 0$. Let $l,m$ be indices such that $i,j,l,m$ are pairwise distinct, the subsequence of $v$ consisting of its $j$th and $m$th entries is $(1,0)$, and the subsequence of $v$ consisting of its $i$th and $l$th entries is $(\mathbf{1},\mathbf{0})$, which we write in bold to distinguish it from the entries indexed by $j,m$. Note that if $j < m$, then we are assuming that $(v_j,v_m) = (1,0)$ while if $m < j$, we are assuming that $(v_m,v_j) = (1,0)$. Let $u$ be obtained from $v$ by swapping its $j$th and $m$th entries, and let $u'$ be obtained from $v$ by swapping its $i$th and $l$th entries. We show that \[
		\begin{tikzcd}
			R(u) \ar[d,"\xi_i"] & R(v) \ar[l,swap,"\xi_j"] \ar[d,"\xi_i"]\\
			R(w) & R(u') \ar[l,swap,"\xi_j"]
		\end{tikzcd}
	\]commutes. Modulo swapping which pair we write as bold, there are three possibilities for how the two subsequences $(1,0)$ and $(\mathbf{1},\mathbf{0})$ are interlaced within $v$. \begin{itemize}[noitemsep]
		\item[$\bullet$] Case $(\mathbf{1},\mathbf{0},1,0)$ or $(1,0,\mathbf{1},\mathbf{0})$. The diagram commutes by far commutativity.
		\item[$\bullet$] Case $(\mathbf{1},1,\mathbf{0},0)$ or $(1,\mathbf{1},0,\mathbf{0})$. The diagram commutes because both $\xi_i\,\xi_j$ and $\xi_j\,\xi_i$ are zero by a computation similar to the one in the second-to-last case in the proof of Lemma~\ref{lem:XChainMap}. 
		\item[$\bullet$] Case $(\mathbf{1},1,0,\mathbf{0})$ or $(1,\mathbf{1},\mathbf{0},0)$. The diagram commutes by a computation similar to the one in the last case in the proof of Lemma~\ref{lem:XChainMap}. 
	\end{itemize}

	Now consider a component map $R(v) \to R(w)$ where $v$ and $w$ differ in exactly two entries, indexed by $l < m$. If $(v_l,v_m) \neq (1,0)$, then the component map of $[\xi_i,x_j] + [x_i,\xi_j] + [\xi_i,\xi_j]$ is zero, so we assume that $(v_l,v_m) = (1,0)$. Without loss of generality, assume that $i < j$. \begin{itemize}[noitemsep]
		\item[$\bullet$] Case $l = i$ and $m = j$. Then $\xi_j = - \xi_i$ as component maps from $R(v)$ to $R(w)$ so $[\xi_i,x_j] + [x_i,\xi_j] = [\xi_i,x_i + x_j]$. The component map of $\xi_i$ is $\theta_i\cdots\theta_{j-2}\,s_{j-1}\,\hat{\theta}_{j-2}\cdots\hat{\theta}_i$ and by reasoning similar to the second case of the proof of Lemma~\ref{lem:XChainMap}, we have \begin{align*}
			\theta_i\cdots\theta_{j-2}\,s_{j-1}\,\hat{\theta}_{j-2}\cdots\hat{\theta}_i\,x_i &= x_j\,\theta_i\cdots\theta_{j-2}\,s_{j-1}\,\hat{\theta}_{j-2}\cdots\hat{\theta}_i + \sum_{\substack{i \leq l \leq j-2\\v_{l+1}=1}} \theta_i\cdots\theta_{j-2}\, s_{j-1}\, \hat{\theta}_{j-2}\cdots\hat{\theta}_{l+1}\,\hat{\theta}_{l-1}\,\cdots\hat{\theta}_{i}\\
			\theta_i\cdots\theta_{j-2}\,s_{j-1}\,\hat{\theta}_{j-2}\cdots\hat{\theta}_i\,x_j &= x_i\,\theta_i\cdots\theta_{j-2}\,s_{j-1}\,\hat{\theta}_{j-2}\cdots\hat{\theta}_i - \sum_{\substack{i\leq l \leq j-2\\v_{l+1}=0}} \theta_i\cdots\theta_{l-1}\,\theta_{l+1}\cdots\theta_{j-2}\,s_{j-1}\,\hat{\theta}_{j-2}\cdots\hat{\theta}_i
		\end{align*}so \begin{align*}
			[\xi_i,x_i + x_j] &= \sum_{\substack{i \leq l \leq j-2\\v_{l+1}=1}} (\theta_i\cdots\theta_{l-1}\,s_l\,\hat{\theta}_{l-1}\cdots\hat{\theta}_{i})\,(\theta_{l+1}\cdots\theta_{j-2}\, s_{j-1}\, \hat{\theta}_{j-2}\cdots\hat{\theta}_{l+1})\\
			&\hspace{50pt} - \sum_{\substack{i\leq l \leq j-2\\v_{l+1}=0}} (\theta_{l+1}\cdots\theta_{j-2}\,s_{j-1}\,\hat{\theta}_{j-2}\cdots\hat{\theta}_{l+1})\,(\theta_i\cdots\theta_{l-1}\,s_l\,\hat{\theta}_{l-1}\cdots\hat{\theta}_i)
		\end{align*}For each $l \in \{i,\ldots,j-2\}$ with $v_{l+1} = 1$, let $u$ be the vertex obtained from $v$ by swapping its $(l+1)$th and $j$th entries. The expression in the first sum indexed by $l$ is precisely the negative of the composite of $\xi_j\colon R(v) \to R(u)$ followed by $\xi_i\colon R(u) \to R(w)$. So the first sum is exactly $-\xi_i\,\xi_j\colon R(v) \to R(w)$. Similar reasoning shows that the second sum is exactly $-\xi_j\,\xi_i\colon R(v) \to R(w)$ so $[\xi_i,x_i+x_j] = [\xi_j,\xi_i]$ as required. 
		\item[$\bullet$] Case $l = i$ and $m < j$. The components of $[x_i,\xi_j]$ and $[\xi_i,\xi_j]$ are zero, and the component of $[\xi_i,x_j]$ is zero by far commutativity.
		\item[$\bullet$] Case $l = i$ and $j < m$. If $v_j = 1$, then $\xi_i = \theta_i\cdots\theta_{j-2}\,s_{j-1}\,\theta_j\cdots\theta_{m-2}\,s_{m-1}\,\hat{\theta}_{m-2}\cdots\hat{\theta}_j\,\partial_{j-1}\,\hat{\theta}_{j-2}\cdots\hat{\theta}_i$ so \begin{align*}
			[\xi_i,x_j] &= - \,\theta_i\cdots\theta_{j-2}\,s_{j-1}\,\theta_j\cdots\theta_{m-2}\,s_{m-1}\,\hat{\theta}_{m-2}\cdots\hat{\theta}_j\,\hat{\theta}_{j-2}\cdots\hat{\theta}_i\\
			&= -\,(\theta_i\cdots\theta_{j-2}\,s_{j-1}\,\hat{\theta}_{j-2}\cdots\hat{\theta}_i)\,(\theta_j\cdots\theta_{m-2}\,s_{m-1}\,\hat{\theta}_{m-2}\cdots\hat{\theta}_j)
		\end{align*}which is precisely the negative of the composite of $\xi_j\colon R(v) \to R(u)$ followed by $\xi_i\colon R(u) \to R(w)$ where $u$ is obtained from $v$ by swapping its $j$th and $m$th entries. There are no other contributions to $\xi_i\,\xi_j\colon R(u) \to R(w)$ and the component map of $[x_i,\xi_j]$ is zero, so $[\xi_i,x_j] + [x_i,\xi_j] + [\xi_i,\xi_j] = 0$ in this case. The case where $v_j = 0$ is handled similarly. 
	\end{itemize}The remaining cases where $i \neq l$ and $j = m$ follow by the same reasoning as the last two cases above. 

	All other component maps of $[\xi_i,x_j] + [x_i,\xi_j] + [\xi_i,\xi_j]$ are easily seen to be zero, so our proof is complete.
\end{proof}

\begin{lem}\label{lem:nilCoxeterRelations}
	The nil-Coxeter relations for $\Delta_1,\ldots,\Delta_{r-1}$ hold. Explicitly, they are $\Delta_i\,\Delta_i = 0$, the braid relation $\Delta_i\,\Delta_{i+1}\,\Delta_i = \Delta_{i+1}\,\Delta_i\,\Delta_{i+1}$, and far commutativity $\Delta_i\,\Delta_j = \Delta_j\,\Delta_i$ for $j\notin \{i,i+1\}$.
\end{lem}
\begin{proof}
	Far commutativity and the relation $\Delta_i\,\Delta_i = 0$ are straightforward to verify. 

	For the braid relation, fix $v_1,\ldots,v_{i-1},v_{i+3},\ldots,v_r\in\{0,1\}$, and let $v^{a,b,c}\in\{0,1\}^r$ be the vertex with the given fixed entries and whose $i$th, $(i+1)$th, and $(i+2)$th entries are $a,b,c \in \{0,1\}$, respectively. The components of $\Delta_i$ and $\Delta_{i+1}$ between these eight vertices for $(a,b,c) \in \{0,1\}^3$ are \[
		\begin{tikzpicture}
			\node[draw,rounded corners,inner sep=2pt] at (0,0) {
				\begin{tikzcd}[column sep=0pt,row sep=10pt]
					& R(v^{011}) \ar[dr,swap,"\textstyle s_i"] & & R(v^{111}) \ar[out=290,in=250,distance=50pt,loop,"\textstyle\partial_i"]\\
					R(v^{001}) \ar[out=70,in=110,distance=50pt,loop,swap,"\textstyle\partial_i"] & & R(v^{101}) &\\
					& R(v^{010}) \ar[dr,swap,"\textstyle s_i"] & & R(v^{110}) \ar[out=290,in=250,distance=50pt,loop,"\textstyle\partial_i"]\\
					R(v^{000}) \ar[out=70,in=110,distance=50pt,loop,swap,"\textstyle\partial_i"] & & R(v^{100}) &
				\end{tikzcd}
			};
			\node[draw,rounded corners,inner sep=2pt] at (6,0) {
				\begin{tikzcd}[column sep=0pt,row sep=10pt]
					& R(v^{011}) \ar[out=170,in=190,distance=50pt,loop,swap,"\textstyle\partial_{i+1}"] & & R(v^{111}) \ar[out=170,in=190,distance=50pt,loop,swap,"\textstyle\partial_{i+1}"]\\
					R(v^{001}) \ar[rd,swap,"\textstyle s_{i+1}"] & & R(v^{101}) \ar[rd,swap,"\textstyle s_{i+1}"] &\\
					& R(v^{010}) & & R(v^{110}) \\
					R(v^{000}) \ar[out=-10,in=10,distance=50pt,loop,swap,"\textstyle\partial_{i+1}"] & & R(v^{100}) \ar[out=-10,in=10,distance=50pt,loop,swap,"\textstyle\partial_{i+1}"] &
				\end{tikzcd} 
			};
		\end{tikzpicture}
	\]All other components of $\Delta_i$ and $\Delta_{i+1}$ having these vertices as their source or target are zero. It is straightforward now to verify that $\Delta_i\,\Delta_{i+1}\,\Delta_i$ and $\Delta_{i+1}\,\Delta_i\,\Delta_{i+1}$ have the same component maps between these eight vertices. There are only four nonzero component maps: $R(v^{111}) \to R(v^{111})$ and $R(v^{000}) \to R(v^{000})$, where the maps are both $\partial_i\,\partial_{i+1}\,\partial_i = \partial_{i+1}\,\partial_i\,\partial_{i+1}$, and $R(v^{011})\to R(v^{110})$ and $R(v^{001}) \to R(v^{100})$, where the maps are $\partial_i\,s_{i+1}\,s_i = s_{i+1}\,s_i\,\partial_{i+1}$ and $s_i\,s_{i+1}\,\partial_i = \partial_{i+1}\,s_i\,s_{i+1}$, respectively.
\end{proof}

\begin{lem}\label{lem:twistedLeibnizRule}
	The relations $\Delta_i\,X_i = \Id + \,X_{i+1}\,\Delta_i$ and $\Delta_i\,X_{i+1} = -\Id + \,X_i\,\Delta_i$ hold.
\end{lem}
\begin{proof}
	It is straightforward to compute that $\Id + \,x_{i+1}\,\Delta_i - \Delta_i\,x_i$ has the following component maps. If $v$ is a vertex with $v_i \neq v_{i+1}$, then its component map $R(v) \to R(v)$ is the identity, and all other component maps are zero. To show that $\Delta_i\,X_i = \Id + \,X_{i+1}\,\Delta_i$, we must show that $\Delta_i\,\xi_i - \xi_{i+1}\,\Delta_i$ has the same component maps. Let $v$ be a vertex. \begin{itemize}[noitemsep]
		\item[$\bullet$] Case $(v_i,v_{i+1}) = (1,1)$. Suppose $j$ is an index for which $i+1 < j$ and $v_j = 0$, and let $w$ be obtained from $v$ by swapping its $(i+1)$th and $j$th entries. It is straightforward to compute that the component maps of $\Delta_i\,\xi_i$ and $\xi_{i+1}\,\Delta_i$ from $R(v)$ to $R(w)$ are both $\theta_{i+1}\cdots\theta_{j-2}\,s_{j-1}\,\hat{\theta}_{j-2}\cdots\hat{\theta}_{i+1}\,\partial_i$, and that all of their other component maps with source $R(v)$ are zero. So $\Delta_i\,\xi_i - \xi_{i+1}\,\Delta_i$ is zero on $R(v)$. 
		\item[$\bullet$] Case $(v_i,v_{i+1}) = (1,0)$. Then $\xi_{i+1}\,\Delta_i$ is zero on $R(v)$. Let $j$ be an index for which $i +1 < j$ and $v_j = 0$, and let $w$ be obtained from $v$ by swapping its $i$th and $j$th entries. Then $\xi_i\colon R(v) \to R(w)$ is $\partial_i\,\theta_{i+1}\cdots\theta_{j-2}\,s_{j-1}\,\hat{\theta}_{j-2}\cdots\hat{\theta}_{i+1}\,s_i$. Since $\Delta_i\colon R(w) \to R(w)$ is $\partial_i$, we see that $\Delta_i\,\xi_i\colon R(v) \to R(w)$ is zero. As for $\Delta_i\,\xi_i\colon R(v) \to R(v)$, let $w$ be obtained from $v$ by swapping its $i$th and $(i+1)$th entries, and note that $\xi_i\colon R(v) \to R(w)$ and $\Delta_i\colon R(w) \to R(v)$ are both equal to $s_i$, so $\Delta_i\,\xi_i\colon R(v) \to R(v)$ is $s_i\,s_i = \Id$. All other components of $\Delta_i\,\xi_i$ with source $R(v)$ are zero. 
		\item[$\bullet$] Case $(v_i,v_{i+1}) = (0,1)$. Suppose $j$ is an index for which $j < i$ and $v_j = 1$, and let $w$ be obtained from $v$ by swapping its $j$th and $i$th entries. It is straightforward to compute that the component maps of $\Delta_i\,\xi_i$ and $\xi_{i+1}\,\Delta_i$ from $R(v)$ to $R(w)$ are $(\partial_i)\,(\theta_{j}\cdots\theta_{i-2}\,s_{i-1}\,\hat{\theta}_{i-2}\cdots\hat{\theta}_j)$ and $(\theta_j\cdots\theta_{i-2}\,s_{i-1}\,s_i\,\partial_{i-1}\,\hat{\theta}_{i-2}\cdots\hat{\theta}_j)\,(s_i)$, respectively, which are equal. All other component maps of $\Delta_i\,\xi_i$ with source $R(v)$ are zero. Letting $w$ be the vertex obtained from $v$ by swapping its $i$th and $(i+1)$th entries, we see that $\Delta_i\colon R(v) \to R(w)$ is $s_i$ and $\xi_{i+1}\colon R(w) \to R(v)$ is $-s_i$ so the component of $\Delta_i\,\xi_i - \xi_{i+1}\,\Delta_i$ from $R(v)$ to $R(v)$ is the identity as required.
		\item[$\bullet$] Case $(v_i,v_{i+1}) = (0,0)$. Let $j$ be an index for which $j < i$ and $v_j = 1$, and let $w$ be obtained from $v$ by swapping its $j$th and $(i+1)$th entries. Then $\xi_{i+1}\,\Delta_i\colon R(v) \to R(w)$ is $(\theta_j\cdots\theta_{i-2}\,\partial_{i-1}\,s_i\,s_{i-1}\,\hat{\theta}_{i-2}\cdots\hat{\theta}_j)\,(\partial_i) = 0$. All other components of $\xi_{i+1}\,\Delta_i$ and $\Delta_i\,X_i$ with source $R(v)$ are zero as well. 
	\end{itemize}
	This completes the proof that $\Delta_i\,X_i = \Id + \,X_{i+1}\,\Delta_i$. Similar reasoning shows that $\Delta_i\,X_{i+1} = -\Id+ \,X_i\,\Delta_i$ as well. 
\end{proof}

\begin{lem}\label{lem:farCommutativity}
	If $j \notin \{i,i+1\}$, then $\Delta_i$ and $X_j$ commute.
\end{lem}
\begin{proof}
	It suffices to show that $[\Delta_i,\xi_j] = 0$. Let $v$ be a vertex, and suppose $l$ is an index for which $j \neq l$ and the subsequence of $v$ consisting of its $j$th and $l$th entries is $(1,0)$. So if $j < l$, then $(v_j,v_l) = (1,0)$ while if $l < j$, then $(v_l,v_j) = (1,0)$. \begin{itemize}[noitemsep]
		\item[$\bullet$] Case $i+1 < \min(j,l)$ or $\max(j,l) < i$. The restrictions of $\Delta_i\,\xi_j$ and $\xi_j\,\Delta_i$ to $R(v)$ are equal by far commutativity. 
		\item[$\bullet$] Case $\min(j,l) < i < i+1 < \max(j,l)$. Let $w$ be obtained from $v$ by swapping its $i$th and $(i+1)$th entries and its $j$th and $l$th entries. If $(v_i,v_{i+1}) = (1,1)$ and $j < l$, then the components of $\Delta_i\,\xi_j$ and $\xi_j\,\Delta_i$ from $R(v)$ to $R(w)$ are \begin{align*}
			\Delta_i\,\xi_j &= (\partial_i)\,(\theta_j\cdots\theta_{i-2}\,\partial_{i-1}\,\partial_i\,\theta_{i+1}\cdots\,\theta_{l-2}\,s_{l-1}\,\hat{\theta}_{l-2}\cdots\,\hat{\theta}_{i+1}\,s_i\,s_{i-1}\,\hat{\theta}_{i-2}\cdots\hat{\theta}_j)\\
			\xi_j\,\Delta_i &= (\theta_j\cdots\theta_{i-2}\,\partial_{i-1}\,\partial_i\,\theta_{i+1}\cdots\,\theta_{l-2}\,s_{l-1}\,\hat{\theta}_{l-2}\cdots\,\hat{\theta}_{i+1}\,s_i\,s_{i-1}\,\hat{\theta}_{i-2}\cdots\hat{\theta}_j)\,(\partial_i)
		\end{align*}
		which are equal by far commutativity and braid relations. All other cases where $v_i = v_{i+1}$ are handled similarly. If $(v_i,v_{i+1}) = (0,1)$ and $j < l$, the two composites are \begin{align*}
			\Delta_i\,\xi_j &= (s_i)\,(\theta_j\cdots\theta_{i-2}\,\partial_{i-1}\,s_i\,\theta_{i+1}\cdots\,\theta_{l-2}\,s_{l-1}\,\hat{\theta}_{l-2}\cdots\,\hat{\theta}_{i+1}\,s_i\,\partial_{i-1}\,\hat{\theta}_{i-2}\cdots\hat{\theta}_j)\\
			\xi_j\,\Delta_i &= (\theta_j\cdots\theta_{i-2}\,s_{i-1}\,\partial_i\,\theta_{i+1}\cdots\,\theta_{l-2}\,s_{l-1}\,\hat{\theta}_{l-2}\cdots\,\hat{\theta}_{i+1}\,\partial_i\,s_{i-1}\,\hat{\theta}_{i-2}\cdots\hat{\theta}_j)\,(s_i)
		\end{align*}
		which are again equal by far commutativity and braid relations. The case where $l < j$ is the same. Finally, if $(v_i,v_{i+1}) = (1,0)$, both composites are zero. 
		\item[$\bullet$] Case $j < i$ and $l \in \{i,i+1\}$. If $(v_i,v_{i+1}) = (0,0)$, then let $u$ be obtained from $v$ by swapping its $j$th and $(i+1)$th entries, and let $w$ be obtained from $v$ by swapping its $j$th and $i$th entries. Then $\Delta_i\,\xi_j\colon R(v) \to R(u)$ is zero while $\xi_j\,\Delta_i\colon R(v) \to R(u)$ is the composite of $\Delta_i\colon R(v) \to R(v)$ followed by $\xi_j\colon R(v) \to R(u)$, which is $(\theta_j\cdots\theta_{i-2}\,\partial_{i-1}\,s_i\,s_{i-1}\,\hat{\theta}_{i-2}\cdots\hat{\theta}_j)\,(\partial_i) = 0$. On the other hand, the $\Delta_i\,\xi_j\colon R(v) \to R(w)$ is the composite of $\xi_j\colon R(v) \to R(u)$ followed by $\Delta_i\colon R(u) \to R(w)$, while $\xi_j\,\Delta_i\colon R(v) \to R(w)$ is the composite of $\Delta_i\colon R(v) \to R(v)$ followed by $\xi_j\colon R(v) \to R(w)$. These two composites are the left- and right-hand sides of the following equation \[
			(s_i)\,(\theta_j\cdots\theta_{i-2}\,\partial_{i-1}\,s_i\,s_{i-1}\,\hat{\theta}_{i-2}\cdots\hat{\theta}_j) = (\theta_j\cdots\theta_{i-2}\,s_{i-1}\,\hat{\theta}_{i-2}\cdots\hat{\theta}_j)\,(\partial_{i}).
		\]If $(v_i,v_{i+1}) = (1,0)$, let $w$ be obtained from $v$ by swapping its $j$th and $(i+1)$th entries. Then $\xi_j\,\Delta_i\colon R(v) \to R(w)$ is zero while $\Delta_i\,\xi_j\colon R(v) \to R(w)$ is $(\partial_i)\,(\theta_j\cdots\theta_{i-2}\,s_{i-1}\,s_i\,\partial_{i-1}\,\hat{\theta}_{i-2}\cdots\hat{\theta}_j) = 0$. If $(v_i,v_{i+1}) = (0,1)$, let $w$ be obtained from $v$ by swapping its $j$th and $i$th entries. Then $[\Delta_i,\xi_j]\colon R(v) \to R(w)$ is zero because \[
			(\partial_i)\,(\theta_j\cdots\theta_{i-2}\,s_{i-1}\,\hat{\theta}_{i-2}\cdots\hat{\theta}_j) = (\theta_j\cdots\theta_{i-2}\,s_{i-1}\,s_i\,\partial_{i-1}\,\hat{\theta}_{i-2}\cdots\hat{\theta}_j)\,(s_i). 
		\]
		\item[$\bullet$] Case $l \in \{i,i+1\}$ and $i+1 < j$. This case is handled in the same way as the previous case. 
	\end{itemize}
	The analysis of these cases completes the argument. 
\end{proof}

\raggedright
\bibliography{freeloop}
\bibliographystyle{alpha}

\vfill

\textit{Department of Mathematics, Princeton University, Princeton NJ 08540}

\textit{School of Mathematics, Institute for Advanced Study, Princeton NJ 08540}

\textit{Email addresses:} \hspace{2pt} {joshuaxw@princeton.edu} \hspace{2pt} {jxwang@ias.edu}

\end{document}